%% file: Lyu_Ouimet_Feng_2026-05-18__R1___as_submitted_.tex
\documentclass[11pt]{article}

\usepackage{amsmath,amssymb,amsthm,graphicx,caption}

\providecommand{\doi}[1]{\href{https://doi.org/#1}{DOI:#1}}
\usepackage{xurl} 
\renewcommand{\doi}[1]{%
 \href{https://doi.org/#1}{\nolinkurl{DOI:#1}}%
}

\usepackage[round]{natbib}
\usepackage[bookmarksnumbered,colorlinks,allcolors=blue]{hyperref}
\usepackage{appendix} 
\usepackage[ruled,vlined]{algorithm2e} 
\usepackage{bm} 
\usepackage{booktabs} 
\usepackage{tabularx}
\usepackage{xcolor}
\usepackage{subcaption} 
\usepackage{float} 
\usepackage{enumerate} 
\usepackage{dsfont} 
\usepackage[top=3cm,bottom=3cm,left=2.5cm,right=2.5cm]{geometry}

\numberwithin{equation}{section}

\theoremstyle{plain}
\newtheorem{theorem}{Theorem}[section]
\newtheorem{proposition}[theorem]{Proposition}

\newtheorem{corollary}[theorem]{Corollary}
\newtheorem{assumption}{Assumption}
\newtheorem{remark}{Remark}

\newcommand{\NN}{{\mathbb{N}}}
\newcommand{\PP}{{\mathsf{P}}}
\newcommand{\EE}{{\mathsf{E}}}
\newcommand{\Var}{\mathsf{Var}}
\newcommand{\Bias}{\mathsf{Bias}}
\newcommand{\MSE}{\mathsf{MSE}}
\newcommand{\IMSE}{\mathsf{IMSE}}
\newcommand{\II}{\mathds{1}}
\newcommand{\bb}[1]{\boldsymbol{#1}}

\title{Asymptotic properties of the multivariate Sz\'{a}sz-Mirakyan estimator for cumulative distribution functions on the nonnegative orthant}

\author{Guanjie Lyu\footnote{Corresponding author at: Department of Community Health and Epidemiology, Faculty of Medicine, Dalhousie University, Canada, E-mail: glyu@dal.ca} ,\ Fr\'ed\'eric Ouimet\footnote{D\'epartement de math\'ematiques et d'informatique, Universit\'e du Qu\'ebec \`a Trois-Rivi\`eres, Trois-Rivi\`eres, Canada, E-mail: frederic.ouimet2@uqtr.ca} \ and\ Cindy Feng\footnote{ Department of Community Health and Epidemiology, Faculty of Medicine, Dalhousie University, Canada, E-mail: cindy.feng@dal.ca} }

\begin{document}
\date{\vspace{-5ex}}

\maketitle

\begin{abstract}
The asymptotic properties of multivariate Sz\'{a}sz–Mirakyan estimators for cumulative distribution functions (cdf) supported on the nonnegative orthant are investigated. Explicit bias and variance expansions are derived on compact subsets of the interior, yielding sharp mean squared error characterizations and optimal smoothing rates. The analysis shows that the proposed Poisson smoothing yields a non-negligible variance reduction relative to the empirical cdf, leading to asymptotic efficiency gains that can be quantified through local and global deficiency measures. The behavior of the estimator near the boundary of its support is examined separately. Under a boundary-layer scaling that preserves nondegenerate Poisson smoothing as the evaluation point approaches the boundary of $[0,\infty)^d$, bias and variance expansions are obtained that differ fundamentally from those in the interior region. In particular, the variance reduction mechanism disappears at leading order, implying that no asymptotically optimal smoothing parameter exists in the boundary regime. Central limit theorems and almost sure uniform consistency are also established. Together, these results provide a unified asymptotic theory for multivariate Sz\'{a}sz--Mirakyan cdf estimation and clarify the distinct roles of smoothing in the interior and boundary regions.
\end{abstract}

\hspace{0mm} 

\section{Introduction}

Modeling and estimating multivariate distributions is a common task in many statistical settings, including structural reliability analysis~\citep{Li2013modeling}, survival studies~\citep{Dabrowska1988kaplan,Dabrowska1989kaplan}, and multivariate pollutant measurements~\citep{Schmidt2003pollutant}, yet it remains technically delicate. On the nonnegative orthant $[0,\infty)^d$, the empirical cumulative distribution function (cdf) is unbiased and uniformly consistent and therefore serves as a natural baseline estimator. However, its step-function form can exhibit unstable finite-sample behavior, motivating the development of smooth cdf estimators that preserve the interpretability and shape constraints of a cdf while offering improved practical performance. At the same time, classical kernel-smoothing techniques face intrinsic difficulties on $[0,\infty)^d$: the challenge comes from the coordinate hyperplanes forming the boundary, i.e., the set where at least one component equals zero, where standard fixed symmetric kernel estimators can suffer from boundary bias (edge effects); see~\cite{MR4899357} for a discussion.

To address boundary effects in kernel-based cdf estimation on bounded or one-sided supports, one can incorporate explicit boundary modifications, such as boundary kernels or asymmetric kernels \citep{Tenreiro2018new,Zhang2020boundary,LafayeDeMicheaux2021study,MR4349488,Mansouri2024CDF}. However, such boundary-specific constructions can add extra tuning and analysis in boundary regions, motivating the consideration of simpler support-adapted alternatives. One such approach is to employ nonnegative-support approximation operators that are intrinsically adapted to $[0,\infty)^d$. In this spirit, Sz\'{a}sz--Mirakyan--type constructions smooth the cdf by averaging with Poisson weights over a data-dependent neighborhood that remains entirely within the support, thereby avoiding boundary leakage and respecting the geometry of the nonnegative orthant \citep{Mirakjan1941,Szasz1950}. This operator-based approach is closely related to Bernstein polynomials in approximation theory \citep{Lorentz1986} and provides a tractable framework for deriving explicit bias--variance expansions and for quantifying efficiency gains relative to the empirical cdf in multivariate settings.

The Sz\'{a}sz--Mirakyan operator has motivated sustained interest in statistical applications. Early work by \citet{Gawronski1980,Gawronski1981} established asymptotic bias and variance properties for density and cdf estimation on $[0,\infty)$, while subsequent analysis by \citet{BouezmarniScaillet2005} demonstrated weak, uniform, and $L_1$ consistency for density estimators and clarified how Poisson-based smoothing mitigates boundary bias at zero. More recently,~\cite{Ouimet2021lecam} established a sharp Le~Cam asymptotic equivalence between Poisson and Gaussian experiments and leveraged this connection to derive variance expansions of Sz\'{a}sz--Mirakyan type estimators. A further development is due to \citet{HanebeckKlar2021}, who introduced a smooth estimator of univariate cdfs on $[0,\infty)$ based on Sz\'asz--Mirakyan operators and showed that Poisson smoothing eliminates boundary bias and yields favorable mean squared error performance across a broad range of scenarios. The present work advances this line of research by extending Sz\'asz--Mirakyan smoothing to the multivariate setting and develops a unified asymptotic theory for cdf estimation on $[0,\infty)^d$. Related extensions of Poisson-based smoothing include the smooth Stute-type estimator of \citet{MR4784228}, which adapts the methodology to censored bivariate data.

In what follows, we introduce the multivariate Sz\'{a}sz--Mirakyan operator and the associated smoothing of cdfs on $[0,\infty)^d$. For $\bb{m} = (m_1,\ldots,m_d)\in\NN^d$, define the product--Poisson weights
\[
P_{\bb{k},\bb{m}}(\bb{x}) = \prod_{j=1}^d \exp(-m_j x_j)\frac{(m_j x_j)^{k_j}}{k_j!}, \qquad \bb{k}\in\NN_0^d,\ \bb{x}\in[0,\infty)^d,
\]
and the Sz\'{a}sz–Mirakyan smoothing of $F$ by
\[
F_{\bb{m}}(\bb{x})\ = \ \sum_{\bb{k}\in\NN_0^d} F\Big(\frac{\bb{k}}{\bb{m}}\Big) \, P_{\bb{k},\bb{m}}(\bb{x}), \qquad \frac{\bb{k}}{\bb{m}} = \Big(\frac{k_1}{m_1},\ldots,\frac{k_d}{m_d}\Big).
\]
Let $\bb{X}_1,\ldots,\bb{X}_n$ be a random sample of independent and identically distributed (iid) observations from a $d$-variate cdf $F$ supported on $[0,\infty)^d$. For $\bb{x} = (x_1,\ldots,x_d)\in[0,\infty)^d$ and a vector of smoothing parameters $\bb{m} = (m_1,\ldots,m_d)\in\mathbb N^d$, the multivariate version of the Sz\'{a}sz--Mirakyan cdf estimator is defined by
\begin{equation}\label{eq:distribution_est}
F_{\bb{m},n}(\bb{x}) = \sum_{\bb{k}\in\mathbb N_0^d} F_n\Bigl(\frac{\bb{k}}{\bb{m}}\Bigr) \, P_{\bb{k},\bb{m}}(\bb{x}).
\end{equation}
Here, $F_n$ denotes the empirical cdf on $[0,\infty)^d$,
\begin{equation}\label{eq:ecdf}
F_n(\bb{x}) = \frac{1}{n}\sum_{i=1}^n \II\!\left\{\bb{X}_i \leqslant \bb{x}\right\} = \frac{1}{n}\sum_{i=1}^n \prod_{j=1}^d \II\!\left\{X_{ij} \leqslant x_j\right\},
\end{equation}
where $\bb{X}_i = (X_{i1},\ldots,X_{id})^{\top}$ and the inequality is understood coordinatewise. Equivalently, let $W_{ij} = \lceil m_j X_{ij}\rceil$ ($i = 1,\ldots,n; j = 1,\ldots,d$), where $\lceil\cdot\rceil$ denotes the ceiling function, i.e., rounding a real number up to the smallest integer greater than or equal to it. Then estimator~\eqref{eq:distribution_est} admits the sample average representation
\begin{equation}\label{eq:distribution_est_2}
F_{\bb{m},n}(\bb{x}) = \frac{1}{n}\sum_{i=1}^n \prod_{j=1}^d \PP\left(\mathrm{Poi}(m_j x_j) \geq W_{ij}\right),
\end{equation}
where $\mathrm{Poi}(\lambda)$ denotes a Poisson random variable with mean $\lambda$.

Throughout the paper, let $\|\cdot\|_{\infty}$ denote the uniform norm and $\|\cdot\|_1$ the $\ell_1$-norm. The notation $u = O(v)$ means that $\limsup |u/v| \leq C < \infty$ as $m\to \infty$ or $n\to \infty$, depending on the context. The positive constant $C$ may depend on the target cdf $F$ and the dimension $d$, but no other variables unless explicitly written as a subscript. In the proof of \texorpdfstring{\hyperref[thm:deficiency-sm]{Theorem~\ref{thm:deficiency-sm}}}{Theorem~\ref{thm:deficiency-sm}}, a common occurrence is a local dependence of the asymptotics on a given point $\bb{x}\in (0,\infty)^d$, in which case one writes $u = O_{\bb{x}}(v)$. Similarly, the notation $u = o(v)$ means that $\lim |u/v| = 0$ as $m\to \infty$ or $n\to \infty$. Subscripts indicate which parameters the convergence rate can depend on. The notation $u\asymp v$ means that $u = O(v)$ and $v = O(u)$ hold simultaneously.

The remainder of the paper is organized as follows.~\hyperref[sec:interior]{Section~\ref{sec:interior}} studies the interior asymptotic properties of the Sz\'{a}sz--Mirakyan estimator, including bias--variance expansions, mean squared error analysis, asymptotic deficiency, and limit theorems. \hyperref[sec:boundary]{Section~\ref{sec:boundary}} investigates the estimator’s behavior near the boundary of $[0,\infty)^d$, deriving bias and variance expansions under a boundary-layer scaling and highlighting the fundamental differences between interior and boundary regimes. \hyperref[sec:simulations]{Section~\ref{sec:simulations}} reports a Monte Carlo simulation study comparing the empirical cdf $F_n$ and the Sz\'{a}sz--Mirakyan estimator $F_{\bb{m},n}$ on compact interior regions and illustrating data-driven smoothing selection via least-squares cross-validation. Finally, \hyperref[sec:conclusion]{Section~\ref{sec:conclusion}} concludes with a discussion of the main findings and directions for future research.

\section{Interior properties}\label{sec:interior}

The multivariate Szász--Mirakyan estimator proposed here admits anisotropic smoothing through coordinate-specific smoothing parameters \(m=(m_1,\ldots,m_d)\). Allowing the components of \(m\) to differ provides greater flexibility by accommodating direction-dependent Poisson smoothing around the evaluation point. Figure~\ref{fig:SM_smoothing} displays iso-level sets of the bivariate product-Poisson weight $P_{\bb{k}, \bb{m}}(\bb{x}) = \prod_{j=1}^{d} \mathrm{e}^{-m_j x_j} (m_j x_j)^{k_j}/k_j!$ at a fixed interior evaluation point $\bb{x} = (x_1, x_2)$, plotted as a function of the rescaled lattice coordinate $\bb{k}/\bb{m} = (k_1/m_1, k_2/m_2)$. In this rescaling the smoothing weight is centred at $\bb{x}$, with component-wise standard deviation of order $\sqrt{x_j / m_j}$, so a smaller $m_j$ corresponds to greater dispersion of the Poisson weights along coordinate $j$. The two panels use $\sqrt{m_1 m_2} = 100$ in both cases, isolating the ratio $m_1/m_2$ as the only quantity that differs between them. In the anisotropic case $(m_1, m_2) = (50, 200)$ the level sets are ellipses elongated along the coordinate with the smaller smoothing parameter; in the isotropic case $m_1 = m_2 = 100$ they degenerate to concentric circles centred at $\bb{x}$, indicating uniform localisation across coordinates. This geometric contrast makes explicit the coordinate-wise structure of the bias and variance expansions developed below: contributions from each coordinate enter additively and are weighted by $1/m_j$, so anisotropic choices of $\bb{m}$ trade localisation in one direction for additional smoothing in another.

\begin{figure}[ht]
\centering
\includegraphics[width=\linewidth]{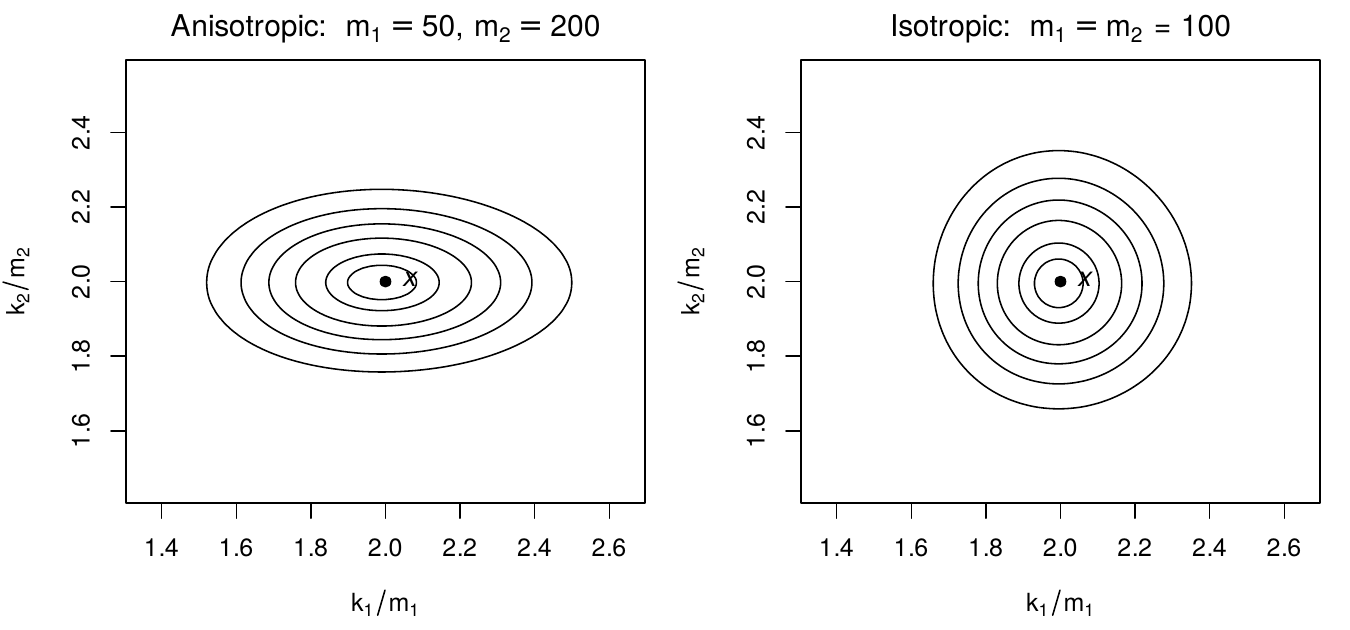}
\caption{Iso-level sets of the bivariate product--Poisson smoothing weights
	\(P_{k,m}(x)\), plotted as a function of the rescaled lattice point \(k/m\)
	at a fixed interior evaluation point \(x=(2,2)\).}
\label{fig:SM_smoothing}
\end{figure}

To study the asymptotic behavior of the Sz\'{a}sz--Mirakyan cdf estimator in the interior of its support, we restrict attention to compact subsets of $(0,\infty)^d$, away from boundary effects. The following regularity condition guarantees that local Taylor expansions of the cdf are valid uniformly in a neighbourhood of the evaluation point, which is essential for deriving sharp bias and variance expansions.

\begin{assumption}\label{ass:C2-vector}
The cdf $F$ is twice continuously differentiable on $(0,\infty)^d$. Moreover, for every compact set $S\subseteq (0,\infty)^d$ there exist constants $\delta_S > 0$ and $M_S < \infty$ such that
\[
\sup_{\bb{y}\in \mathcal N_{\delta_S}(S)}\max_{1 \leq i \leq d}\big|\partial_{x_i} F(\bb{y})\big|\ \leq \ M_S, \qquad \sup_{\bb{y}\in \mathcal N_{\delta_S}(S)}\max_{1 \leq i,j \leq d}\big|\partial_{x_i x_j}^2 F(\bb{y})\big|\ \leq \ M_S,
\]
where
\[
\mathcal N_{\delta_S}(S) := \Big\{\bb{y}\in(0,\infty)^d:\ \exists \, \bb{z}\in S\ \text{with}\ \|\bb{y}-\bb{z}\|_{1} \leq \delta_S\Big\}.
\]
\end{assumption}
Under Assumption~\ref{ass:C2-vector}, the Poisson smoothing induced by the Sz\'{a}sz--Mirakyan operator admits a second--order expansion with uniformly controlled remainder on compact subsets of $(0,\infty)^d$. In particular, the centering property of the Poisson kernel eliminates first--order bias terms, while bounded second derivatives of $F$ determine the leading contribution to the smoothing bias. The results below derive explicit bias and variance expansions for $F_{m,n}$, which provide the foundation for the mean squared error analysis, optimal smoothing rates, and the deficiency comparisons developed in subsequent subsections.

\begin{remark}[Remark on the dependence of remainder terms on compact sets]
Throughout Section~\ref{sec:interior}, whenever an expansion is stated uniformly for $\bb{x}\in S$, with $S\subseteq (0,\infty)^d$ compact, the corresponding $O$- and $o$-terms may depend on this fixed set $S$. We keep this dependence implicit, rather than writing $O_S(\cdot)$ and $o_S(\cdot)$ throughout, in order to avoid unnecessary notational clutter. More concretely, for every such compact set $S$ one can choose numbers $b_i\in(0,1)$, $i = 1,\ldots,d$, such that
\[
S\subseteq \prod_{i=1}^d [b_i,b_i^{-1}], \qquad \bb{b} = (b_1,\ldots,b_d).
\]
The uniform constants and convergence rates used below may be controlled in terms of these coordinatewise lower and upper bounds, together with the relevant local bounds and moduli of continuity of the derivatives of $F$ on a slightly enlarged rectangle, for example $\prod_{i=1}^d [b_i/2,2b_i^{-1}]$. Thus, in the interior theory developed here, the relevant information about $S$ is its coordinatewise separation from the boundary of $[0,\infty)^d$ and its coordinatewise upper size, as encoded by the $b_i$'s. Finer geometric properties of the set, such as convexity or the precise shape of its boundary, do not enter the remainder bounds in an essential way.
\end{remark}

\subsection{Bias and variance expansions}

The following proposition provides a uniform second--order expansion of the smoothed cdf on interior compact sets, which isolates the leading contribution of the operator bias.

\begin{proposition}\label{prop:SM-bias-vector}
Suppose Assumption~\ref{ass:C2-vector} holds. Then, for every compact set $S\subseteq (0,\infty)^d$, as $m_{\min} = \min_{1 \leq j \leq d} m_j\to\infty$ (with no asymptotic regime imposed on $n$),
\[
F_{\bb{m}}(\bb{x}) = F(\bb{x}) + \frac{1}{2}\sum_{j=1}^d \frac{x_j}{m_j} \, \partial_{x_jx_j}^2F(\bb{x}) + o(m_{\min}^{-1}),
\]
uniformly for $\bb{x}\in S$.
\end{proposition}

\begin{proof}
The proof is given in \ref{app}.
\end{proof}

Proposition~\ref{prop:SM-bias-vector} identifies the deterministic smoothing bias of the Sz\'{a}sz--Mirakyan operator on interior compact sets. Combining this bias with the sampling variability of the empirical Sz\'{a}sz--Mirakyan estimator yields the joint bias--variance expansion for $F_{\boldsymbol m,n}$ stated next.

\begin{theorem}\label{thm:szasz-bias-var}
Suppose Assumption~\ref{ass:C2-vector} holds. Then, for every compact $S\subseteq (0,\infty)^d$,
\[
\Bias\left(F^{}_{\bb{m}, n}(\bb{x})\right) = \EE\big[F^{}_{\bb{m}, n}(\bb{x})\big]-F(\bb{x}) = \frac{1}{2}\sum_{j=1}^d \frac{x_j}{m_j} \, \partial_{x_jx_j}^2F(\bb{x}) + o\big(m_{\min}^{-1}\big), \qquad \forall \bb{x}\in S,
\]
and
\[
\Var\big(F^{}_{\bb{m}, n}(\bb{x})\big) = {n^{-1}} \, \sigma^2(\bb{x}) -{n^{-1}} \, V_{}(\bb{x};\bb{m}) + O\big(n^{-1}m_{\min}^{-1}\big), \qquad \forall \bb{x}\in S,
\]
as $m_{\min} = \min_{1 \leq j \leq d} m_j\to\infty$ and $n\to\infty$, where
\[
\sigma^2(\bb{x}) = F(\bb{x})\{1-F(\bb{x})\}, \qquad V_{}(\bb{x};\bb{m}) = \sum_{j=1}^d m_j^{-1/2} \, \partial_{x_j}F(\bb{x}) \, \sqrt{\frac{x_j}{\pi}}.
\]
\end{theorem}

\begin{proof}
The proof is given in \ref{app}.
\end{proof}

The bias--variance expansion derived above implies that the leading bias of $F_{\bb{m},n}(\bb{x})$ coincides with that of the deterministic smoothed estimator $F_{\bb{m}}(\bb{x})$, while the variance differs from that of the empirical cdf by a negative correction term of order $n^{-1} m_{\min}^{-1/2}$. This correction depends on the local slope of $F$ through its first--order partial derivatives and captures the variance--stabilizing effect of Poisson smoothing. Consequently, the estimator exhibits a nontrivial bias--variance trade--off, i.e., larger values of $\bb{m}$ reduce the smoothing bias but weaken the variance reduction, whereas smaller values of $\bb{m}$ strengthen variance stabilization at the expense of increased bias.

\subsection{Mean squared error and optimal smoothing}

This subsection studies the mean squared error of the Sz\'{a}sz--Mirakyan estimator by combining the bias and variance expansions obtained in the previous subsection. The resulting expressions allow us to identify optimal smoothing rates.

\begin{corollary}\label{cor:szasz-mse}
Assume Assumption~\ref{ass:C2-vector} and take equal smoothing levels $m_1 = \cdots = m_d\equiv m$. Then, for every compact set $S\subseteq (0,\infty)^d$ and uniformly for $\bb{x}\in S$, as $n\to\infty$ and $m\to\infty$,
\[
\MSE\big(F_{\bb{m},n}(\bb{x})\big) = n^{-1} \, \sigma^2(\bb{x}) -{n^{-1}}{m^{-1/2}} \, V_{\mathrm{}}(\bb{x}) + m^{-2} \, B^2(\bb{x}) + O\big(n^{-1}m^{-1}\big) + o\big(m^{-2}\big),
\]
where
\[
\sigma^2(\bb{x}) = F(\bb{x})\{1-F(\bb{x})\},\qquad V_{\mathrm{}}(\bb{x}) = \sum_{j=1}^d \partial_{x_j}F(\bb{x}) \, \sqrt{\frac{x_j}{\pi}}, \qquad B(\bb{x}) = \frac{1}{2}\sum_{j=1}^d x_j \, \partial^2_{x_jx_j}F(\bb{x}).
\]
In particular, if $V_{\mathrm{}}(\bb{x}) > 0$ and $B(\bb{x}) \neq 0$, the asymptotically optimal choice of $m$ with respect to the MSE at $\bb{x}$ is
\[
m_{\mathrm{opt}}(\bb{x}) = n^{2/3}\left[\frac{4 \, B^2(\bb{x})}{V_{\mathrm{}}(\bb{x})}\right]^{2/3},
\]
in which case
\[
\MSE\big(F_{m_{\mathrm{opt}},n}(\bb{x})\big) = {n^{-1}} \, \sigma^2(\bb{x}) - n^{-4/3} \, \frac{3}{4}\left[\frac{V^4_{\mathrm{}}(\bb{x})}{4 \, B^2(\bb{x})}\right]^{1/3} + o\big(n^{-4/3}\big).
\]
\end{corollary}
Corollary~\ref{cor:szasz-mse} makes explicit the competing contributions of variance reduction and smoothing bias in the mean squared error of the Sz\'{a}sz--Mirakyan estimator under isotropic smoothing. The negative term of order $n^{-1}m^{-1/2}$ reflects the variance stabilization induced by Poisson smoothing, while the positive term of order $m^{-2}$ captures the second--order bias identified in Proposition~\ref{prop:SM-bias-vector}. Balancing these two terms yields an optimal smoothing rate $m_{\mathrm{opt}}(x)\asymp n^{2/3}$, which coincides with the critical regime where neither bias nor variance dominates.

The next result extends this analysis to an integrated setting, allowing for a global assessment of smoothing performance over compact subsets of $(0,\infty)^d$. The following provides a similar result for the integrated MSE.

\begin{corollary}\label{thm:szasz-mise}
Assume Assumption~\ref{ass:C2-vector} and take equal smoothing levels $m_1 = \cdots = m_d\equiv m$. Then, for every compact set $S\subseteq (0,\infty)^d$, as $n\to\infty$ and $m\to\infty$,
\begin{align*}
\IMSE_S\big(F_{\bb{m},n}\big)
&= n^{-1}\int_S \sigma^2(\bb{x}) \, \mathrm{d}\bb{x} - n^{-1} m^{-1/2}\int_S V_{\mathrm{}}(\bb{x}) \, \mathrm{d}\bb{x} \\
&\qquad + m^{-2}\int_S B^2(\bb{x}) \, \mathrm{d}\bb{x} + O\big(n^{-1} m^{-1}\big) + o\big(m^{-2}\big).
\end{align*}
In particular, if
\[
\int_S V_{\mathrm{}}(\bb{x}) \, \mathrm{d}\bb{x} > 0
\qquad\text{and}\qquad
\int_S B^2(\bb{x}) \, \mathrm{d}\bb{x} > 0,
\]
the asymptotically optimal choice of $m$ with respect to $\IMSE_S$ is
\[
m_{\mathrm{opt}} = n^{2/3}\left[ \frac{4\displaystyle\int_S B^2(\bb{x}) \, \mathrm{d}\bb{x}} {\displaystyle\int_S V_{\mathrm{}}(\bb{x}) \, \mathrm{d}\bb{x}} \right]^{2/3},
\qquad
\bb{m}_{\mathrm{opt}} := m_{\mathrm{opt}}\bb{1}_d,
\]
in which case, as $n\to\infty$,
\[
\IMSE_S\big(F_{\bb{m}_{\mathrm{opt}},n}\big) = n^{-1}\int_S \sigma^2(\bb{x}) \, \mathrm{d}\bb{x} - n^{-4/3} \, \frac{3}{4} \left[ \frac{\big\{\displaystyle\int_S V_{\mathrm{}}(\bb{x}) \, \mathrm{d}\bb{x}\big\}^{4}} {4\displaystyle\int_S B^2(\bb{x}) \, \mathrm{d}\bb{x}} \right]^{1/3} + o\big(n^{-4/3}\big).
\]
\end{corollary}
The optimal rate obtained here coincides with the critical smoothing regime identified in the pointwise analysis and reflects the balance between the variance reduction induced by Poisson smoothing and the second--order bias of the operator.

The resulting rate is comparable to that obtained for kernel cdf estimators and Bernstein polynomial estimators under analogous smoothness conditions \citep[see, e.g.,][]{Jones1990,Leblanc2012a}, while retaining the advantage of respecting the nonnegative orthant support of the distribution. In particular, unlike kernel methods, the Sz\'{a}sz--Mirakyan estimator avoids boundary corrections on $[0,\infty)^d$ \citep[see, e.g.,][]{HanebeckKlar2021}, and, in contrast to Bernstein estimators, naturally accommodates unbounded support \citep[see, e.g.,][]{Gawronski1980}, making it a flexible alternative for cdf estimation on the nonnegative orthant. These MSE and IMSE analyses set the stage for the efficiency comparison with the empirical cdf developed in the next subsection.

\subsection{Asymptotic deficiency of the empirical cdf}

This subsection compares the efficiency of the empirical cdf and the Sz\'{a}sz--Mirakyan estimator through local and global deficiency measures. These quantities characterize the effective sample size required by the empirical cdf to achieve the same mean squared error performance as the smoothed estimator.

\begin{theorem}\label{thm:deficiency-sm}
For every compact set $S\subseteq (0,\infty)^d$ such that $\int_S\sigma^2(\bb{x}) \, \mathrm{d}\bb{x} > 0$, fix $\bb{x}\in S$ such that $0 < F(\bb{x}) < 1$. Define the local effective empirical sample size associated with the Sz\'{a}sz--Mirakyan estimator $F_{\bb{m},n}$ at $\bb{x}$ by
\[
L^{\mathrm{SM}}(n,\bb{x}) := \min\Big\{k\in\NN: \MSE\big(F_k(\bb{x})\big) \leq \MSE\big(F_{\bb{m},n}(\bb{x})\big)\Big\},
\]
and define the global effective empirical sample size associated with $F_{\bb{m},n}$ on $S$ by
\[
G^{\mathrm{SM}}_S(n) := \min\Big\{k\in\NN:\IMSE_S(F_k) \leq \IMSE_S(F_{\bb{m},n})\Big\},
\]
where $F_k$ denotes the empirical cdf based on a sample of size $k$. The corresponding local and global deficiencies of the empirical cdf are
\[
L^{\mathrm{SM}}(n,\bb{x})-n
\qquad\text{and}\qquad
G^{\mathrm{SM}}_S(n)-n,
\]
respectively. Assume Assumption~\ref{ass:C2-vector} and take equal smoothing levels $m_1 = \cdots = m_d\equiv m$. Throughout this theorem, all limits are taken as $n\to\infty$ and $m\to\infty$. Then, if $nm^{-2}\to 0$,
\[
L^{\mathrm{SM}}(n,\bb{x}) = n\{1 + o_{\bb{x}}(1)\} \qquad \text{and} \qquad G^{\mathrm{SM}}_S(n) = n\{1 + o(1)\}.
\]
Moreover,
\begin{enumerate}[(i)]
\item \textit{(Large-$m$, bias-negligible regime.)} if $m n^{-2/3}\to \infty$ and $m n^{-2}\to 0$, then
\begin{equation}\label{eq:def-sm-regime-a}
\begin{aligned}
L^{\mathrm{SM}}(n,\bb{x})-n
&= \frac{n}{m^{1/2}} \left\{ \frac{V(\bb{x})}{\sigma^2(\bb{x})} + o_{\bb{x}}(1) \right\}, \\
G^{\mathrm{SM}}_S(n)-n
&= \frac{n}{m^{1/2}} \left\{ \frac{\int_S V(\bb{x}) \, \mathrm{d}\bb{x}}{\int_S\sigma^2(\bb{x}) \, \mathrm{d}\bb{x}} + o(1) \right\};
\end{aligned}
\end{equation}

\item \textit{(Critical regime.)} if $m n^{-2/3}\to c$, where $c$ is a positive constant, then
\begin{equation}\label{eq:def-sm-regime-b}
\begin{aligned}
L^{\mathrm{SM}}(n,\bb{x})-n
&= n^{2/3} \left\{ c^{-1/2}\frac{V(\bb{x})}{\sigma^2(\bb{x})} - c^{-2}\frac{B^2(\bb{x})}{\sigma^2(\bb{x})} + o_{\bb{x}}(1) \right\}, \\
G^{\mathrm{SM}}_S(n)-n
&= n^{2/3} \left\{ c^{-1/2}\frac{\int_S V(\bb{x}) \, \mathrm{d}\bb{x}}{\int_S \sigma^2(\bb{x}) \, \mathrm{d}\bb{x}} - c^{-2}\frac{\int_S B^2(\bb{x}) \, \mathrm{d}\bb{x}}{\int_S \sigma^2(\bb{x}) \, \mathrm{d}\bb{x}} + o(1) \right\}.
\end{aligned}
\end{equation}
\end{enumerate}
Here, $V(\bb{x})$ and $B(\bb{x})$ are defined in Corollary~\ref{cor:szasz-mse}. Consequently, in the large-$m$ regime, $L^{\mathrm{SM}}(n,\bb{x})-n\to\infty$ if $V(\bb{x})>0$, and $G^{\mathrm{SM}}_S(n)-n\to\infty$ if $\int_S V(\bb{x}) \, \mathrm{d}\bb{x}>0$. In the critical regime, $L^{\mathrm{SM}}(n,\bb{x})-n\to\infty$ if $c^{3/2}V(\bb{x})>B^2(\bb{x})$, and $G^{\mathrm{SM}}_S(n)-n\to\infty$ if
\[
c^{3/2}\int_S V(\bb{x}) \, \mathrm{d}\bb{x} > \int_S B^2(\bb{x}) \, \mathrm{d}\bb{x}.
\]
Under these conditions, the empirical cdf is asymptotically deficient with respect to the Sz\'{a}sz--Mirakyan cdf estimator.
\end{theorem}

\begin{proof}
The proof is given in \ref{app}.
\end{proof}

The large-$m$ and critical regimes reflect the relative magnitude of the variance-reduction and bias terms in the local mean squared error expansion. Since larger values of $m$ make the Poisson kernel more concentrated, the first regime is described as a large-$m$, bias-negligible regime rather than as a high-smoothing regime. In this regime, the smoothing bias is asymptotically negligible relative to the variance reduction, so that the efficiency gain of the Sz\'{a}sz--Mirakyan estimator is driven primarily by variance stabilization. In contrast, the critical regime corresponds to a balance between bias and variance corrections, where both contributions enter at the same asymptotic order. This balance produces a phase transition in the efficiency behaviour, with the deficiency growth rate depending explicitly on both curvature and slope characteristics of the underlying cdf.

A small-$m$ regime, characterized by $m n^{-2/3}\to 0$, is not used here for positive deficiency conclusions. At a point $\bb{x}$ such that $B(\bb{x})\neq 0$, the smoothing bias term dominates the variance-reduction term, and the mean squared error of the Sz\'{a}sz--Mirakyan estimator exceeds that of the empirical cdf asymptotically. Likewise, in the integrated setting, the same conclusion follows when $\int_S B^2(\bb{x})\,\mathrm{d}\bb{x}>0$. If the relevant leading bias coefficient vanishes, the expansion displayed above does not by itself determine the sign of the MSE difference, and higher-order terms would be needed.

\subsection{Limit distributions and uniform consistency}

This subsection establishes the asymptotic distributional behavior and uniform consistency of the Sz\'{a}sz--Mirakyan cdf estimator on interior compact sets. These results complement the mean squared error and deficiency analyses by showing that smoothing does not alter the first--order limiting distribution of the estimator under suitable growth conditions on the smoothing parameters.

\begin{theorem}\label{thm:sm-cdf-clt-interior}
Suppose Assumption~\ref{ass:C2-vector} holds. Fix $\bb{x}\in (0,\infty)^d$ such that $0 < F(\bb{x}) < 1$. Then, as $n\to\infty$, the following convergence results hold.
\begin{enumerate}[(a)]
\item If $m_{\min} \to \infty$, then
\begin{equation}\label{eq:sm-clt-interior-mean}
\sqrt{n} \, \Big\{ F_{\bb{m},n}(\bb{x}) - \EE\big[F_{\bb{m},n}(\bb{x})\big] \Big\} \xrightarrow{d} \mathcal N\Big( 0, F(\bb{x})\{1-F(\bb{x})\} \Big).
\end{equation}

\item If, in addition, $\sqrt{n} \, m_{\min}^{-1} \to 0$, then
\[
\sqrt{n} \, \big\{ F_{\bb{m},n}(\bb{x}) - F(\bb{x}) \big\} \xrightarrow{d} \mathcal N\Big( 0, F(\bb{x})\{1-F(\bb{x})\} \Big).
\]
\end{enumerate}
\end{theorem}

\begin{proof}
The proof is given in \ref{app}.
\end{proof}

The results show that the Sz\'{a}sz--Mirakyan estimator shares the same first--order asymptotic distribution as the empirical cdf at interior points. Centering by the smoothed mean yields a central limit theorem under the minimal requirement that the smoothing parameters diverge, while additional undersmoothing ensures that the bias becomes asymptotically negligible and centering by the true cdf is valid. As a result, Poisson smoothing improves finite-sample efficiency without affecting the classical $\sqrt{n}$ convergence rate or limiting variance. Almost sure uniform consistency of the Sz\'{a}sz--Mirakyan estimator on compact subsets of the nonnegative orthant is provided in the following theorem.

\begin{theorem}\label{thm:sm-cdf-uniform}
Assume that $F$ is continuous on every compact subset of $[0,\infty)^d$. Then, for any $L > 0$,
\[
\sup_{\bb{x}\in[0,L]^d} \big|F_{\bb{m},n}(\bb{x})-F(\bb{x})\big| \xrightarrow{a.s.} 0,
\]
as $n\to\infty$ and $m_{\min}\to\infty$. Consequently, $F_{\bb{m},n}$ converges to $F$ almost surely, uniformly on compact subsets of $[0,\infty)^d$.
\end{theorem}

\begin{proof}
The proof is given in \ref{app}.
\end{proof}

\section{Near boundary properties}\label{sec:boundary}

This section investigates the behavior of the Sz\'{a}sz--Mirakyan cdf estimator in neighborhoods of the boundary of its support, where standard interior asymptotic expansions no longer apply.

To characterize the leading bias and variance terms in this regime, we impose additional smoothness conditions on the cdf in a neighborhood of the boundary. In particular, the following assumption ensures that $F$ admits well-defined second--order expansions along boundary faces, which allows the effect of smoothing to be isolated as the evaluation point approaches the boundary at a rate determined by the smoothing parameters.

\begin{assumption}\label{ass:szasz-C2-boundary}
Let $\mathcal X = [0,\infty)^d$. Assume that the cdf $F$ admits a density $f$ and that $F$ is twice continuously differentiable on a neighborhood of the boundary $\partial\mathcal X := \{\bb{x}\in\mathcal X:\min_{1 \leq j \leq d} x_j = 0\}$; that is, there exists $\delta > 0$ such that
\[
\max_{1 \leq j \leq d}\ \sup_{\bb{x}:\ \min_j x_j \leq \delta} \, \bigl|\partial_{x_j} F(\bb{x})\bigr| < \infty,\qquad \max_{1 \leq i,j \leq d}\ \sup_{\bb{x}:\ \min_j x_j \leq \delta} \, \bigl|\partial_{x_i x_j}^2 F(\bb{x})\bigr| < \infty.
\]
\end{assumption}

The following theorem characterizes the bias and variance expansions of $F_{\bb{m}, n}$ in a neighborhood of the boundary. The scaling $x_j = \lambda_j/m_j$ represents the natural boundary-layer regime of the Sz\'{a}sz--Mirakyan operator, ensuring that the Poisson kernel remains nondegenerate as the evaluation point approaches the boundary. This rate uniquely separates interior behavior from degenerate smoothing and leads to asymptotic bias and variance expansions that differ fundamentally from those obtained away from the boundary.

\begin{theorem}\label{thm:boundary-bias}\label{thm:Boundary_variance_cdf}
Fix $L > 0$ and, for each $\bb{m} = (m_1,\dots,m_d)$ with $m_{\min} = \min_{1 \leq j \leq d} m_j$, set
\[
\bb{x} = \Big(\frac{\lambda_1}{m_1},\dots,\frac{\lambda_d}{m_d}\Big), \qquad (\lambda_1,\dots,\lambda_d)\in[0,L]^d.
\]
Under Assumption~\ref{ass:szasz-C2-boundary}, as $n\to\infty$ and $m_{\min}\to\infty$,
\begin{align}
\Bias\big(F_{\bb{m},n}(\bb{x})\big)
&= \EE\big[F_{\bb{m},n}(\bb{x})\big] - F(\bb{x}) = \frac{1}{2} \sum_{j=1}^d \frac{\lambda_j}{m_j^2} \, \partial^2_{x_j x_j} F\big(x^{(j,0)}\big) + o\big(m_{\min}^{-2}\big), \label{thm:bias.boundary} \\
\Var\big(F_{\bb{m},n}(\bb{x})\big)
&= \frac{1}{n} \, \sigma^2(\bb{x}) + O\big(n^{-1}m_{\min}^{-1}\big), \label{thm:var.boundary}
\end{align}
uniformly over $(\lambda_1,\dots,\lambda_d)\in[0,L]^d$, where $x^{(j,0)} := (x_1,\dots,x_{j-1},0,x_{j + 1},\dots,x_d)$ and $\sigma^2(\bb{x}) = F(\bb{x})\bigl\{1-F(\bb{x})\bigr\}$. Note that if $x_j = 0$ for some $j$, then $F_{\bb{m},n}(\bb{x}) = 0$ almost surely, so the bias and variance are trivially zero in this case.
\end{theorem}

\begin{proof}
The proof is given in \ref{app}.
\end{proof}

An immediate implication of Theorem~\ref{thm:boundary-bias} is that the Sz\'{a}sz--Mirakyan estimator retains a second--order bias structure as the evaluation point approaches the boundary of the support. In particular, no first--order boundary bias arises, and the leading bias term remains of order $m_{\min}^{-2}$, reflecting the support--preserving and centered nature of the Poisson smoothing operator. Interpreting the Poisson smoothing parameter through the effective bandwidth scaling $h \asymp m_{\min}^{-1}$, the resulting boundary bias is of order $O(h^2)$. This rate is substantially smaller than the $O(h)$ boundary bias typically exhibited by classical kernel cdf estimators on compact supports in the absence of correction, and it coincides with the $O(h^2)$ boundary bias achieved by boundary kernel cdf estimators such as those studied in~\cite{Zhang2020boundary}. The result highlights the intrinsic boundary adaptivity of the proposed estimator. Note also that the variance term has the same asymptotic order as that of the boundary kernel cdf estimator studied in~\cite{Zhang2020boundary}.

Combining the bias and variance expansions obtained above yields a precise characterization of the mean squared error of the Sz\'{a}sz--Mirakyan estimator in the boundary region. The following corollary summarizes the leading MSE behavior when the evaluation point approaches the boundary at the rate $x_j = \lambda_j/m_j$.

\begin{corollary}\label{cor:boundary-mse}
Fix $L > 0$ and let
\[
\bb{x} = \Big(\frac{\lambda_1}{m_1},\ldots,\frac{\lambda_d}{m_d}\Big), \qquad (\lambda_1,\ldots,\lambda_d)\in[0,L]^d,
\]
with $m_{\min} = \min_{1 \leq j \leq d} m_j$. Suppose Assumption~\ref{ass:szasz-C2-boundary} holds. Then, as $n\to\infty$ and $m_{\min}\to\infty$, the mean squared error satisfies
\begin{align*}
\MSE \, \big(F_{\bb{m},n}(\bb{x})\big)
&= \Var\big(F_{\bb{m},n}(\bb{x})\big) + \Bias\big(F_{\bb{m},n}(\bb{x})\big)^2 \\[0.4em]
&= \frac{1}{n} \, F(\bb{x})\bigl\{1-F(\bb{x})\bigr\} + \frac{1}{4}\Bigg( \sum_{j=1}^d \frac{\lambda_j}{m_j^{2}} \, \partial_{x_jx_j}^2 F\big(\bb{x}^{(j,0)}\big) \Bigg)^{2} + O\big(n^{-1}m_{\min}^{-1}\big) + o\big(m_{\min}^{-4}\big),
\end{align*}
uniformly for $(\lambda_1,\ldots,\lambda_d)\in[0,L]^d$. If $x_j = 0$ for some $j$, then under our absolute continuity assumptions $F_{\bb{m},n}(\bb{x}) = 0$ almost surely, so $\MSE(F_{\bb{m},n}(\bb{x})) = 0$ in this corner case.
\end{corollary}

Note that even when the smoothing parameters are chosen isotropically, $m_1 = \cdots = m_d\equiv m$, there is no asymptotically optimal choice of $m$ for minimizing the mean squared error in this boundary regime. Indeed, the leading variance term remains $n^{-1}F(x)\{1-F(x)\}$ and is independent of $m$, while the leading bias contribution is of strictly smaller order, namely $O(m^{-4})$. As a result, the usual bias--variance tradeoff collapses near the boundary: increasing $m$ does not reduce the dominant variance term but only suppresses higher--order remainder terms. Consequently, the asymptotic MSE is minimized by taking $m$ as large as allowed within the asymptotic framework, and no finite optimal smoothing parameter emerges in this regime.

This contrasts sharply with the interior region studied in~\hyperref[sec:interior]{Section~\ref{sec:interior}}, where the MSE admits a nontrivial bias--variance trade--off. There, Poisson smoothing produces a variance reduction of order $n^{-1}m^{-1/2}$ while the smoothing bias contributes a term of order $m^{-2}$, leading to an interior-optimal rate $m_{\mathrm{opt}}\asymp n^{2/3}$. Near the boundary, the effective smoothing scale becomes $x_j = \lambda_j/m_j\to 0$, and the variance reduction term that drives the interior efficiency gains vanishes from the leading expansion. As a result, the boundary behavior is dominated by the baseline binomial variance term, and smoothing primarily affects only higher--order corrections rather than inducing an interior-type phase transition.

\section{Simulation study}\label{sec:simulations}

This section reports a Monte Carlo study comparing the empirical cdf $F_n$ and the multivariate Sz\'{a}sz--Mirakyan estimator $F_{\bb{m},n}$, defined in Equations~\eqref{eq:distribution_est} and~\eqref{eq:ecdf}, respectively. The primary goal is to illustrate, on compact interior regions, the integrated risk behavior predicted by the bias--variance and $\IMSE$ expansions derived in Section~\ref{sec:interior}. In particular, Corollary~\ref{thm:szasz-mise} shows that, on compact subsets of the interior and under critical isotropic smoothing levels \(m\asymp n^{2/3}\), the Szász--Mirakyan smoothing yields a second-order reduction in integrated mean squared error relative to the empirical cdf. We also report boundary-layer experiments to illustrate how the finite-sample behavior changes when the evaluation region approaches the boundary of the nonnegative orthant.

\subsection{Models, evaluation region, and performance criteria}\label{sec:sim_model}

Fix a small constant $\delta > 0$ and define the compact interior region
\[
S_{\delta} := [\delta,\delta^{-1})^d \subseteq (0,\infty)^d, \qquad |S_{\delta}| = (\delta^{-1}-\delta)^d.
\]
Except for the boundary-layer comparison reported in Figure~\ref{fig:boundary}, the integrated criteria below are computed over \(S_\delta\), which is bounded away from the boundary \(\partial[0,\infty)^d\) and hence falls under the interior asymptotic regime of Section~2. For Figure~\ref{fig:boundary}, we additionally compute the same criterion over the boundary-layer region \(B_n\) defined in Section~4.3. Throughout, $\delta = 0.05$. The sample sizes are
\[
n\in\{25,50,100,200,400\}.
\]

\paragraph{Data generating mechanisms.} We consider two representative multivariate distributions on $[0,\infty)^d$ with smooth cdf on $(0,\infty)^d$. For simplicity, we only report results for $d = 2$.
\begin{enumerate}[(M1)]
\item \textit{(Independent Gamma, isotropic scale.)} Let $X_1,\ldots,X_d$ be independent with
\[
X_j\sim \mathrm{Gamma}(\alpha,\beta), \qquad (\alpha,\beta) = (2,1), \qquad j = 1,\ldots,d,
\]
so that $F(\bb{x}) = \prod_{j=1}^d F_{\Gamma}(x_j;\alpha,\beta)$ is available in closed form via the regularized lower incomplete gamma function.

\item \textit{(Dependent Clayton copula with Gamma marginals.)} Let $(U_1,\ldots,U_d)$ follow a Clayton copula with parameter $\theta > 0$,
\[
C_{\theta}(\bb{u}) = \Bigg(\sum_{j=1}^d u_j^{-\theta}-d + 1\Bigg)^{-1/\theta}, \qquad \bb{u}\in(0,1]^d,
\]
and set $X_j = F_{\Gamma}^{-1}(U_j;\alpha,\beta)$ with $(\alpha,\beta) = (2,1)$. Then
\[
F(\bb{x}) = C_{\theta}\Big(F_{\Gamma}(x_1;\alpha,\beta),\ldots,F_{\Gamma}(x_d;\alpha,\beta)\Big), \qquad \bb{x}\in[0,\infty)^d,
\]
with $\theta = 2$ in our implementation.
\end{enumerate}

\paragraph{Estimators compared.} For each Monte Carlo replication, we compute:
\begin{enumerate}[(i)]
\item the empirical cdf $F_n(\bb{x})$;
\item the Sz\'{a}sz--Mirakyan estimator $F_{\bb{m}^{\star},n}(\bb{x})$ with $\bb{m}^{\star}$ selected by least-squares cross-validation (Section~\ref{subsec:sim-lscv} below).
\end{enumerate}

\paragraph{Integrated squared error and Monte Carlo risk.} For any cdf estimator $\widehat F$, define the integrated squared error (ISE) over $S_{\delta}$ by
\begin{equation}\label{eq:sim-ise}
\mathrm{ISE}_{S_{\delta}}(\widehat F) := \int_{S_{\delta}}\big\{\widehat F(\bb{x})-F(\bb{x})\big\}^2 \, \mathrm{d} \bb{x}.
\end{equation}
The corresponding integrated mean squared error (IMSE) is $\IMSE_{S_{\delta}}(\widehat F) := \EE[\mathrm{ISE}_{S_{\delta}}(\widehat F)]$. We approximate $\mathrm{ISE}_{S_{\delta}}(\widehat F)$ numerically using a quasi--Monte Carlo (QMC) rule. Specifically, let $\{\bb{u}_g\}_{g = 1}^G\subseteq [0,1]^d$ be a low-discrepancy point set and map it to $S_{\delta}$ via
\[
\bb{x}_g = \delta \, \bb{1}_d + (\delta^{-1}-\delta) \, \bb{u}_g, \qquad g = 1,\ldots,G,
\]
so that
\[
\mathrm{ISE}_{S_{\delta}}(\widehat F) \approx \frac{|S_{\delta}|}{G}\sum_{g = 1}^G\big\{\widehat F(\bb{x}_g)-F(\bb{x}_g)\big\}^2.
\]
In our implementation we take $G = 2^{12} = 4096$ points for $d = 2$, which yields stable numerical integration error at the scales reported in the tables below.

\paragraph{Monte Carlo protocol.} For each $(n,\text{model})$ pair, we run $N_{\mathrm{MC}}$ independent replications (with $N_{\mathrm{MC}} = 100$ in our implementation) and summarize the distribution of $\mathrm{ISE}_{S_{\delta}}$ by its mean, median, interquartile range, and variance. We also report the distribution of the LSCV-selected smoothing level $\bb{m}^{\star}$.

\subsection{Smoothing parameter selection via least-squares cross-validation}\label{subsec:sim-lscv}

The smoothing vector $\bb{m} = (m_1,\ldots,m_d)\in\NN^d$ is selected by minimizing a least-squares cross-validation criterion that targets $\IMSE_{S_{\delta}}$. Write $\widehat F_{\bb{m}} := F_{\bb{m},n}$ for brevity. Expanding the integrated squared error in Equation~\eqref{eq:sim-ise} yields
\[
\mathrm{ISE}_{S_{\delta}}(\widehat F_{\bb{m}}) = \int_{S_{\delta}}\widehat F_{\bb{m}}(\bb{x})^2 \, \mathrm{d} \bb{x} - 2\int_{S_{\delta}}\widehat F_{\bb{m}}(\bb{x}) \, F(\bb{x}) \, \mathrm{d} \bb{x} + \int_{S_{\delta}}F(\bb{x})^2 \, \mathrm{d} \bb{x}.
\]
The last term does not depend on $\bb{m}$, so we drop it. Using $F(\bb{x}) = \EE[\II\!\{\bb{X} \leq \bb{x}\}]$ and Fubini's theorem, we obtain the identity
\[
\int_{S_{\delta}}\widehat F_{\bb{m}}(\bb{x}) \, F(\bb{x}) \, \mathrm{d} \bb{x} = \EE\Bigg[\int_{S_{\delta}}\widehat F_{\bb{m}}(\bb{x}) \, \II\!\{\bb{X} \leq \bb{x}\} \, \mathrm{d} \bb{x}\Bigg].
\]
This motivates the leave-one-out estimator and the following LSCV criterion (up to an additive constant independent of $\bb{m}$):
\begin{equation}\label{eq:lscv-sm}
\mathrm{LSCV}(\bb{m}) = \int_{S_{\delta}}\widehat F_{\bb{m}}(\bb{x})^2 \, \mathrm{d} \bb{x} - \frac{2}{n}\sum_{i=1}^n \int_{S_{\delta}}\widehat F_{\bb{m}}^{(-i)}(\bb{x}) \, \II\!\{\bb{X}_i \leq \bb{x}\} \, \mathrm{d} \bb{x},
\end{equation}
where $\widehat F_{\bb{m}}^{(-i)}$ is the Sz\'{a}sz--Mirakyan estimator computed from the sample with the $i$th observation removed.

\paragraph{Fast leave-one-out evaluation.} Using the sample-average representation~\eqref{eq:distribution_est_2}, define
\[
\psi_{i,\bb{m}}(\bb{x}) := \prod_{j=1}^d \PP\big(\mathrm{Poi}(m_j x_j) \geq W_{ij}\big), \qquad W_{ij} = \lceil m_j X_{ij}\rceil.
\]
Then $\widehat F_{\bb{m}}(\bb{x}) = n^{-1}\sum_{i=1}^n \psi_{i,\bb{m}}(\bb{x})$ and
\[
\widehat F_{\bb{m}}^{(-i)}(\bb{x}) = \frac{1}{n-1}\sum_{\ell \ne i}\psi_{\ell,\bb{m}}(\bb{x}) = \frac{n}{n-1}\widehat F_{\bb{m}}(\bb{x}) - \frac{1}{n-1}\psi_{i,\bb{m}}(\bb{x}).
\]
This identity avoids recomputing the estimator from scratch for each leave-one-out sample and makes the numerical minimization of Equation~\eqref{eq:lscv-sm} feasible.

\paragraph{Numerical implementation of LSCV.} We evaluate both integrals in Equation~\eqref{eq:lscv-sm} on the same QMC grid $\{\bb{x}_g\}_{g = 1}^G$ used for $\mathrm{ISE}_{S_{\delta}}$, yielding
\begin{equation}\label{eq:lscv-sm-qmc}
\widehat{\mathrm{LSCV}}(\bb{m}) = \frac{|S_{\delta}|}{G}\sum_{g = 1}^G \widehat F_{\bb{m}}(\bb{x}_g)^2 - \frac{2}{n}\sum_{i=1}^n \frac{|S_{\delta}|}{G}\sum_{g = 1}^G \widehat F_{\bb{m}}^{(-i)}(\bb{x}_g) \, \II\!\{\bb{X}_i \leq \bb{x}_g\}.
\end{equation}
The selected smoothing vector is
\[
\bb{m}^{\star} \in \operatorname{argmin}_{\bb{m}\in \mathcal M_n}\widehat{\mathrm{LSCV}}(\bb{m}),
\]
where $\mathcal M_n$ is a finite search set described below.

\paragraph{Search domain.} Since Corollary~\ref{thm:szasz-mise} suggests that the IMSE--optimal smoothing levels are of order $n^{2/3}$ in the interior regime (under isotropic smoothing), we restrict the search to the data-dependent range
\[
m_j \in \{m_{\min},m_{\min} + 1,\ldots,m_{\max}(n)\}, \qquad m_{\max}(n) := \min\big(\lfloor c \, n^{2/3}\rfloor,\ m_{\mathrm{cap}},\ n\big),
\]
with default constants $(m_{\min},m_{\mathrm{cap}},c) = (5,500,3)$. To keep the multivariate optimization stable, we minimize Equation~\eqref{eq:lscv-sm-qmc} over $\mathcal M_n$ using a short coordinate-descent scheme initialized at the isotropic pilot $\bb{m}^{(0)} = m^{(0)}\bb{1}_d$, where $m^{(0)}\in\operatorname{argmin}_{m\in\{m_{\min},\ldots,m_{\max}(n)\}}\widehat{\mathrm{LSCV}}(m\bb{1}_d)$. We then perform two full passes over coordinates $j = 1,\ldots,d$, updating $m_j$ by one-dimensional minimization while holding the other coordinates fixed, and set $\bb{m}^{\star}$ to the final iterate.

\bigskip
\begin{algorithm}[H]
\DontPrintSemicolon
\caption{Coordinate-descent LSCV selection of $\bb{m}$}\label{alg:lscv}
\KwIn{Sample $\{\bb{X}_i\}_{i=1}^n$; QMC grid $\{\bb{x}_g\}_{g = 1}^G\subseteq S_{\delta}$; search range $\{m_{\min},\ldots,m_{\max}(n)\}$}
\KwOut{Selected smoothing vector $\bb{m}^{\star}$}
Compute $m^{(0)}\in \operatorname{argmin}_{m}\widehat{\mathrm{LSCV}}(m\bb{1}_d)$ and set $\bb{m}^{(0)} = m^{(0)}\bb{1}_d$
\For{$t = 1,2$}{
\For{$j = 1,\ldots,d$}{
Update $m_j^{(t)}\in\operatorname{argmin}_{m\in\{m_{\min},\ldots,m_{\max}(n)\}} \widehat{\mathrm{LSCV}}(m_1^{(t)},\ldots,m_{j-1}^{(t)},m,m_{j + 1}^{(t-1)},\ldots,m_d^{(t-1)})$
}
}
Set $\bb{m}^{\star} = \bb{m}^{(2)}$
\end{algorithm}

\begin{remark}[On adaptive selection of the smoothing parameters]
The theoretical results in Sections~\ref{sec:interior}--\ref{sec:boundary} are stated for deterministic smoothing vectors $\bb{m}$, whereas our simulation study uses the data-driven LSCV selector $\bb{m}^{\star}$ described above. A rigorous asymptotic theory for such adaptive selectors is possible in principle, but would require additional arguments that are distinct from the bias--variance expansions developed here. In the interior regime, the natural route would be to restrict $\bb{m}$ to a finite grid $\mathcal M_n\subseteq \mathbb N^d$, establish a uniform stochastic approximation of the LSCV criterion to the corresponding integrated risk over $\mathcal M_n$, and then prove an oracle-type inequality showing that the selected estimator attains the risk of the best deterministic choice in the grid, up to a negligible remainder. Such a result would be expected to select smoothing levels of the same order as the theoretically optimal ones, namely $m_j\asymp n^{2/3}$ in the interior setting considered here, possibly with coordinate-dependent constants in the anisotropic case. Proving adaptive second-order outperformance of the empirical cdf is more delicate, however, because the improvement identified in Corollary~\ref{thm:szasz-mise} occurs at the smaller order $n^{-4/3}$ on top of the common leading $n^{-1}$ term. Near the boundary, the boundary-layer analysis shows that the leading variance-reduction mechanism disappears, so one should not expect an adaptive smoothing rule to recover the same efficiency gain as in the interior. A full oracle theory for the LSCV-selected multivariate Sz\'asz--Mirakyan estimator is therefore an interesting direction for future work, but lies beyond the scope of the present paper.
\end{remark}

\subsection{Results and discussion}

Tables~\ref{tab:sim-model1}--\ref{tab:sim-model2} summarize the Monte Carlo distribution of $\mathrm{ISE}_{S_{\delta}}$ for the empirical cdf and the LSCV-selected Sz\'{a}sz--Mirakyan estimator, across sample sizes $n\in\{25,50,100,200,400\}$. In addition, Table~\ref{tab:sim-mstar} reports the empirical scaling of the selected smoothing levels $\bb{m}^{\star}$. Figure~\ref{fig:sim-ise-loglog} provides a visual comparison of the mean ISE across $n$ on a log--log scale.

\paragraph{Expected asymptotic patterns.} Because $S_{\delta}$ is bounded away from the boundary, Corollary~\ref{thm:szasz-mise} predicts that:
\begin{enumerate}[(a)]
\item the dominant term of $\IMSE_{S_{\delta}}$ is of order $n^{-1}$ for both estimators;
\item under critical smoothing $m\asymp n^{2/3}$, the Sz\'{a}sz--Mirakyan estimator exhibits an additional negative correction of order $n^{-4/3}$, implying a second--order IMSE improvement over $F_n$;
\item the LSCV-selected smoothing levels should grow approximately at rate $n^{2/3}$ (up to a model-dependent constant), in agreement with the interior bias--variance balance.
\end{enumerate}
To make these effects visible at finite $n$, we also report the scaled improvement
\[
\Delta_n := n^{4/3}\left\{
\widehat{\textsf{IMSE}}_{D}(F_n)
-
\widehat{\textsf{IMSE}}_{D}(F_{m^\star,n})
\right\},
\qquad D\in\{S_\delta,B_n\},
\]
where $\widehat{\IMSE}_{D}$ denotes the Monte Carlo average of $\mathrm{ISE}_{D}$.

\input{sim_model1.tex}

\input{sim_model2.tex}

\begin{figure}[H]
\centering
\includegraphics[width=0.78\linewidth]{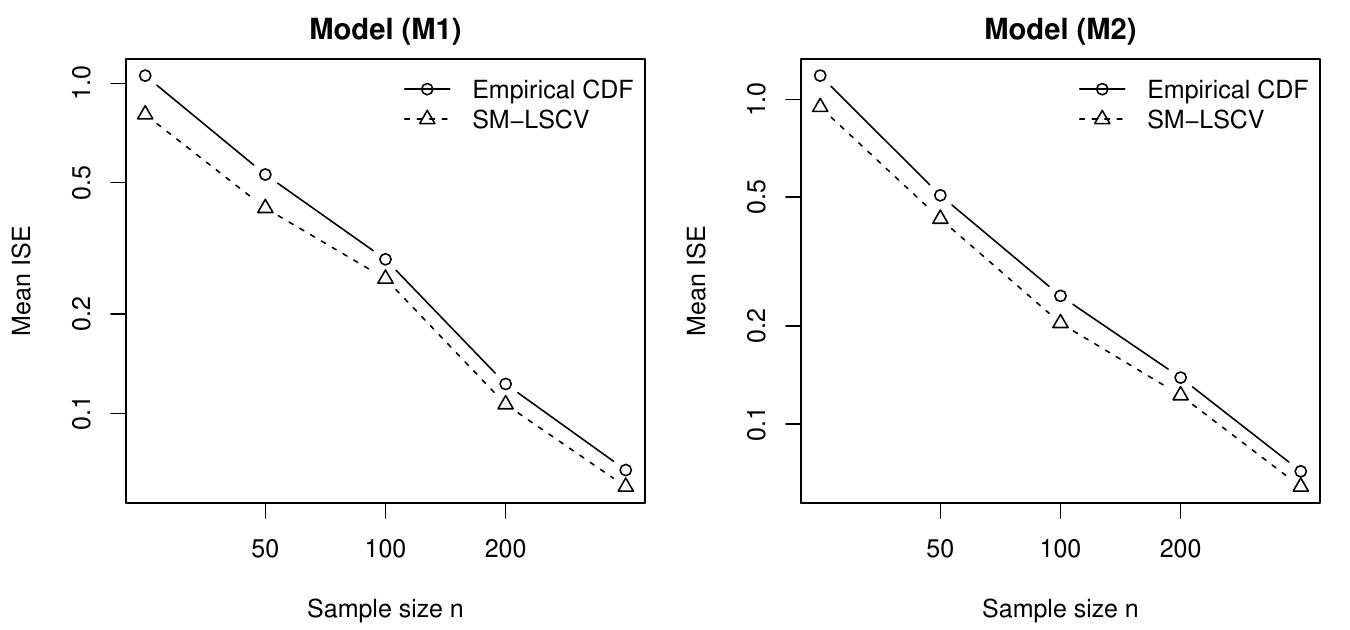}
\caption{Mean $\mathrm{ISE}_{S_{\delta}}$ (Monte Carlo average) versus sample size $n$ on a log--log scale: empirical cdf versus Sz\'{a}sz--Mirakyan with LSCV-selected smoothing.}
\label{fig:sim-ise-loglog}
\end{figure}

\begin{figure}[H]
	\centering
	\includegraphics[width=0.8\linewidth, height=0.4\linewidth]{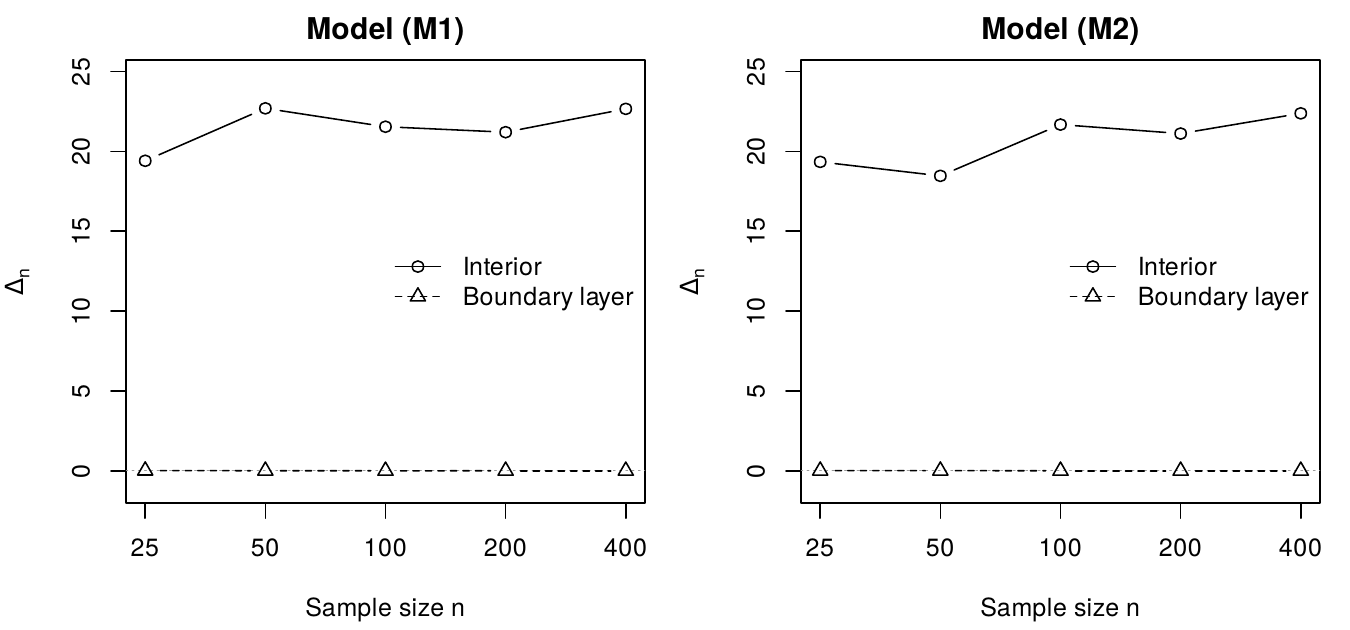}
\caption{Scaled Monte Carlo improvement
	\(\Delta_n\) as a function of the sample size \(n\),
	for models (M1) and (M2). The solid curve corresponds to the compact interior region,
	where the Szász--Mirakyan estimator with LSCV-selected smoothing exhibits a positive
	second-order improvement over the empirical cdf. The dashed curve corresponds to the
	boundary-layer region, where the improvement is negligible, illustrating the disappearance
	of the interior variance-reduction effect near the boundary.}
	\label{fig:boundary}
\end{figure}

\input{sim_mstar.tex}

Figure~\ref{fig:boundary} reports the scaled Monte Carlo improvement $\Delta_n$ on the same integrated squared error scale used in Tables~\ref{tab:sim-model1}--\ref{tab:sim-model2}, separately for the compact interior region $S_\delta$ and the boundary-layer region $B_n = \{\bm x \in [0, \delta^{-1}]^2 : \min_{1 \le j \le 2} x_j \le n^{-2/3}\}$, where $\delta = 0.05$ as in Section~\ref{sec:sim_model}. The thickness \(n^{-2/3}\) is motivated by the critical interior smoothing rate \(m\asymp n^{2/3}\) identified in Corollary~\ref{thm:szasz-mise}. Stabilization at a positive constant is the empirical signature of the $n^{-4/3}$ second-order improvement predicted by Corollary~\ref{thm:szasz-mise}: the multiplicative factor $n^{4/3}$ exactly cancels the theoretical rate, so a genuine interior gain manifests as a horizontal asymptote. The boundary-layer curves, by contrast, are visually indistinguishable from the zero reference line at this scale; the actual values lie between approximately $10^{-2}$ at $n = 25$ and $10^{-5}$ at $n = 400$, three to six orders of magnitude below their interior counterparts and trending toward zero as $n$ grows. This collapse is consistent with the boundary-layer behavior described in Corollary~\ref{cor:boundary-mse}: under the boundary-layer scaling $x_j = \lambda_j / m_j$, the variance-reduction term $V(\bm x; \bm m)$ that drives the interior efficiency gain disappears from the leading expansion, so the integrated squared error of $F_{\bm m, n}$ matches that of $F_n$ to leading order and the rescaled difference $\Delta_n$ vanishes asymptotically.

\section{Concluding remarks}\label{sec:conclusion}

This paper has developed a comprehensive asymptotic theory for the multivariate Sz\'{a}sz–Mirakyan cdf estimator on the nonnegative orthant. By deriving sharp bias and variance expansions, we have shown that Poisson smoothing induces a nontrivial variance reduction in the interior of the support, leading to improved mean squared error performance and asymptotic deficiency of the empirical cdf under appropriate smoothing regimes. The asymptotic analysis and simulation identify a critical smoothing rate at which bias and variance effects balance, yielding optimal pointwise and integrated MSE rates comparable to those of classical kernel and Bernstein-type estimators, while naturally respecting the support constraints.

A key contribution of this work is the explicit distinction between interior and boundary behavior. By introducing a boundary-layer scaling that preserves nondegenerate Poisson smoothing as the evaluation point approaches the boundary, we show that the bias-variance tradeoff near the boundary differs fundamentally from that in the interior. In this regime, the variance reduction mechanism responsible for efficiency gains away from the boundary disappears at leading order, implying that no asymptotically optimal smoothing parameter exists locally. Nevertheless, under the boundary-layer scaling, the proposed estimator attains the same boundary bias and variance rates as boundary-corrected kernel cdf estimators. This contrast highlights an intrinsic feature of Poisson-based operators and clarifies the role of smoothing in cdf estimation on unbounded supports.

Future work may focus on establishing the asymptotic properties of the corresponding multivariate Sz\'{a}sz–Mirakyan density and regression estimators. Another important direction is the extension of the proposed framework to more complex data settings, including censored observations, missing data mechanisms, triangular arrays of observations \citep[see, e.g.,][]{MR4335173,MR4597723}, and dependent or longitudinal data structures.

\appendix

\begin{appendices}

\renewcommand{\thesection}{Appendix~\Alph{section}}

\section{Proofs of results}\label{app}

\renewcommand{\thesection}{\Alph{section}}

\subsection{Proof of \texorpdfstring{\hyperref[prop:SM-bias-vector]{Proposition~\ref{prop:SM-bias-vector}}}{Proposition~\ref{prop:SM-bias-vector}}}

Fix a compact set $S\subseteq (0,\infty)^d$ and $\bb{x}\in S$. Since $S\subseteq (0,\infty)^d$ is compact, let $\delta_S > 0$ be small enough such that $\mathcal N_{\delta_S}(S)\subseteq (0,\infty)^d$. Note that
\[
F_{\bb{m}}(\bb{x}) = \sum_{\bb{k}\in\NN_0^d} F\Big(\frac{\bb{k}}{\bb{m}}\Big) P_{\bb{k},\bb{m}}(\bb{x}) = \EE\Big[F\Big(\frac{\bb{K}}{\bb{m}}\Big)\Big],
\]
where $K_1,\ldots,K_d$ are independent with $K_j\sim{\rm Poi}(m_j x_j)$. Define $\Delta_j := K_j/m_j-x_j$ and $\bb{\Delta} := (\Delta_1,\ldots,\Delta_d)^{\top}$. The second–order mean value theorem yields, for each $\bb{k}$ such that $\bb{k}/m\in \mathcal{N}_{\delta_S}(S)$,
\[
F\Big(\frac{\bb{k}}{\bb{m}}\Big)-F(\bb{x}) = \sum_{i=1}^d \Big(\frac{k_i}{m_i}-x_i\Big)\partial_{x_i}F(\bb{x}) + \frac{1}{2}\sum_{i,j = 1}^d \Big(\frac{k_i}{m_i}-x_i\Big)\Big(\frac{k_j}{m_j}-x_j\Big)\partial_{x_i x_j}^2F(\bb{\xi}_{\bb{k}}),
\]
for some $\bb{\xi}_{\bb{k}}$ on the segment joining $\bb{x}$ and $\bb{k}/\bb{m}$. Let $\eta = \eta_{\bb{m}}\in(0,\delta_S]$ be deterministic with
\[
\eta_{\bb{m}}\downarrow 0, \qquad \frac{m_{\min}\eta_{\bb{m}}^2}{\log m_{\min}}\to\infty \qquad (m_{\min}\to\infty),
\]
for instance $\eta_{\bb{m}} = \delta_S\wedge m_{\min}^{-1/4}$ (throughout, $a\wedge b = \min(a,b)$), and set
\[
A := \left\{\Big\|\frac{\bb{K}}{\bb{m}}-\bb{x}\Big\|_{1} \leq \eta\right\},\qquad A^c := \left\{\Big\|\frac{\bb{K}}{\bb{m}}-\bb{x}\Big\|_{1} > \eta\right\}.
\]
Taking expectation with respect to $P_{\bb{k},\bb{m}}(\bb{x})$ gives
\[
\EE\big[\{F(\bb{K}/\bb{m})-F(\bb{x})\}\mathbf \II_A\big] = \sum_{i=1}^d \EE[\Delta_i \II_A] \, \partial_{x_i}F(\bb{x}) + \frac{1}{2} \sum_{i,j = 1}^d \EE\Big[\Delta_i\Delta_j \, \partial_{x_i x_j}^2F(\bb{\xi}_{\bb{K}})\II_A\Big].
\]
We add and subtract the entries of the fixed Hessian $H(\bb{x}) = \{\partial_{x_i x_j}^2F(\bb{x})\}_{i,j}$ to obtain
\[
\begin{aligned}
\EE\big[\{F(\bb{K}/\bb{m})-F(\bb{x})\}\mathbf \II_A\big]
&= \sum_{i=1}^d \EE[\Delta_i \II_A] \, \partial_{x_i}F(\bb{x}) + \frac{1}{2} \sum_{i,j = 1}^d \EE\Big[\Delta_i\Delta_j \, \partial_{x_i x_j}^2F(\bb{x})\II_A\Big] + R_{\bb{m}}^{(A)}(\bb{x}),
\end{aligned}
\]
where
\begin{equation*}
R_{\bb{m}}^{(A)}(\bb{x}) = \frac{1}{2} \sum_{i,j = 1}^d \EE\Big[\Delta_i\Delta_j\big\{\partial_{x_i x_j}^2F(\bb{\xi}_{\bb{K}})-\partial_{x_i x_j}^2F(\bb{x})\big\}\II_A\Big]
\end{equation*}
If the event $A$ occurs, then the entire segment between $\bb{x}$ and $\bb{K}/\bb{m}$ lies in the tube $\mathcal N_{\delta_S}(S)$, hence $\bb{\xi}_{\bb{K}}\in \mathcal N_{\delta_S}(S)$. By Assumption~\ref{ass:C2-vector}, each Hessian entry is uniformly continuous on $\mathcal N_{\delta_S}(S)$; let $\omega_S(\cdot)$ denote a common modulus of continuity there. Then on $A$,
\[
\max_{i,j}\big|\partial_{x_i x_j}^2F(\bb{\xi}_{\bb{K}})-\partial_{x_i x_j}^2F(\bb{x})\big| \leq \omega_S\left(\Big\|\frac{\bb{K}}{\bb{m}}-\bb{x}\Big\|_{1}\right) \leq \omega_S(\eta).
\]
Since $\EE[\Delta_i] = 0$, $\EE[\Delta_i\Delta_j] = 0$ for $i \ne j$, and $\EE[\Delta_j^2] = \Var(K_j/m_j) = x_j/m_j$. Therefore, by the Cauchy--Schwarz inequality and the Poisson moment identities,
\begin{align*}
|R_{\bb{m}}^{(A)}(\bb{x})| &\leq \frac{1}{2} \, \omega_S(\eta)\sum_{i,j = 1}^d \EE|\Delta_i\Delta_j| \leq C \, \omega_S(\eta)\sum_{i,j = 1}^d \big(\EE[\Delta_i^2]\big)^{1/2}\big(\EE[\Delta_j^2]\big)^{1/2} \\
&= C \, \omega_S(\eta)\sum_{i,j = 1}^d \Big(\frac{x_i}{m_i}\Big)^{1/2}\Big(\frac{x_j}{m_j}\Big)^{1/2} = O\big(m_{\min}^{-1}\big) \, \omega_S(\eta),
\end{align*}
uniformly for $\bb{x}\in S$, where $C$ depends only on $S$. Since $\omega_S(\eta_{\bb{m}})\to 0$ as $m_{\min}\to\infty$, it follows that $|R_{\bb{m}}^{(A)}(\bb{x})| = o(m_{\min}^{-1})$ uniformly for $\bb{x} \in S$.

Return to the raw split
\begin{align*}
F_{\bb{m}}(\bb{x})-F(\bb{x})
&= \EE\big[\{F(\bb{K}/\bb{m})-F(\bb{x})\}\mathbf \II_A\big] + \EE\big[\{F(\bb{K}/\bb{m})-F(\bb{x})\}\mathbf \II_{A^c}\big] \\
&= \Gamma_{\mathrm{near}} + \Gamma_{\mathrm{far}}.
\end{align*}
On $A$ we already obtained
\[
\Gamma_{\mathrm{near}} = \EE\big[\nabla F(\bb{x})^{\top} \bb{\Delta} \, \mathbf \II_A\big] + \frac{1}{2} \, \EE\big[\bb{\Delta}^{\top} H(\bb{x})\bb{\Delta} \, \mathbf \II_A\big] + R_{\bb{m}}^{(A)}(\bb{x}), \qquad |R_{\bb{m}}^{(A)}(\bb{x})| \leq C \, \omega_S(\eta) \, m_{\min}^{-1}.
\]
Add and subtract the complementary quadratic piece with the {fixed} Hessian $H(\bb{x})$,
\[
F_{\bb{m}}(\bb{x})-F(\bb{x}) = \EE\big[\nabla F(\bb{x})^{\top} \bb{\Delta} \, \mathbf \II_A\big] + \frac{1}{2} \, \EE\big[\bb{\Delta}^{\top} H(\bb{x})\bb{\Delta}\big] + R_{\bb{m}}^{(A)}(\bb{x}) + T_{\mathrm{far}},
\]
where
\[
T_{\mathrm{far}} = \Gamma_{\mathrm{far}}-\frac{1}{2} \, \EE\big[\bb{\Delta}^{\top} H(\bb{x})\bb{\Delta} \, \mathbf \II_{A^c}\big].
\]
Further, since $F$ is a cdf, $|F(\bb{K}/\bb{m})-F(\bb{x})| \leq 1$, hence
\[
|\Gamma_{\mathrm{far}}| \leq \PP(A^c).
\]
Because $\bb{x}\in S$ and $F$ is twice continuously differentiable, the entries of $H(\bb{x})$ are bounded on $S$ by some $B_S < \infty$. For $K_{\ell}\sim\mathrm{Poi}(\mu_{\ell})$ one has $\EE[(K_{\ell}-\mu_{\ell})^4] = \mu_{\ell} + 3\mu_{\ell}^2$. With $\mu_{\ell} = m_{\ell} x_{\ell}$,
\[
\EE|\Delta_{\ell}|^4 = \frac{\mu_{\ell} + 3\mu_{\ell}^2}{m_{\ell}^4} = \frac{x_{\ell}}{m_{\ell}^3} + \frac{3x_{\ell}^2}{m_{\ell}^2} = O(m_{\ell}^{-2})\qquad \text{uniformly for }\bb{x}\in S,
\]
hence $\EE[|\Delta_{\ell}|^4]^{1/4} = O(m_{\ell}^{-1/2})$. So by H\"older's inequality with exponents $(4,4,2)$,
\[
\Big|\EE\big[\bb{\Delta}^{\top} H(\bb{x})\bb{\Delta} \, \mathbf \II_{A^c}\big]\Big| \leq B_S\sum_{i,j = 1}^d \EE[|\Delta_i|^4]^{1/4}\EE[|\Delta_j|^4]^{1/4} \, \PP(A^c)^{1/2} = O(m_{\min}^{-1}) \, \PP(A^c)^{1/2}.
\]
With $A = \{\|\bb{K}/\bb{m}-\bb{x}\|_{1} \leq \eta\}$, the union bound and the two–sided Poisson tail inequality~\citep[see][]{Canonne2016short}
\[
\PP\{|X-\lambda| \geq t\} \leq 2\exp\left(-\frac{t^2}{2(\lambda + t)}\right), \qquad X\sim\mathrm{Poi}(\lambda),\ t > 0,
\]
give, uniformly for $\bb{x}\in S$ (where $x_j\in[a,b]$ for some $0 < a \leq b < \infty$),
\[
\begin{aligned}
\PP(A^c)
&= \PP\left(\Big\|\frac{\bb{K}}{\bb{m}}-\bb{x}\Big\|_{1} > \eta\right) \\
&\leq \sum_{j=1}^d \PP\left(\Big|\frac{K_j}{m_j}-x_j\Big| > \eta/d\right) \leq 2d \, \exp\left(-c_{\bb{m}} \, m_{\min}\right), \qquad c_{\bb{m}} := \frac{(\eta_{\bb{m}}/d)^{2}}{2(b + \eta_{\bb{m}}/d)} > 0.
\end{aligned}
\]
By the choice of $\eta_{\bb{m}}$,
\[
\frac{c_{\bb{m}}m_{\min}}{\log m_{\min}} \geq \frac{m_{\min}\eta_{\bb{m}}^2}{2d^2(b + \delta_S/d)\log m_{\min}} \to \infty,
\]
and hence $\PP(A^c) = o(m_{\min}^{-q})$ for every fixed $q > 0$. In particular, $\PP(A^c) = o(m_{\min}^{-1})$ and $\PP(A^c)^{1/2} = o(m_{\min}^{-1/2})$. Therefore,
\[
|T_{\mathrm{far}}| \leq \PP(A^c) + O(m_{\min}^{-1})\PP(A^c)^{1/2} \leq o(m_{\min}^{-1}).
\]
Further, note that,
\begin{equation*}
\EE\big[\nabla F(\bb{x})^{\top} \bb{\Delta} \, \II_A\big] = 0- \EE\big[\nabla F(\bb{x})^{\top} \bb{\Delta} \, \II_{A^c}\big].
\end{equation*}
Then, by the Cauchy--Schwarz inequality,
\[
\Big|\EE\big[\nabla F(\bb{x})^{\top} \bb{\Delta} \, \II_{A^c}\big]\Big| \leq \|\nabla F(\bb{x})\|_2 \, \EE\big[\|\bb{\Delta}\|_2 \, \II_{A^c}\big] \leq \|\nabla F(\bb{x})\|_2 \, \left\{\EE\|\bb{\Delta}\|_2^2\right\}^{1/2}\PP(A^c)^{1/2}.
\]
Moreover,
\[
\EE\|\bb{\Delta}\|_2^2 = \sum_{j=1}^d \EE[\Delta_j^2] = \sum_{j=1}^d \frac{x_j}{m_j} \leq \frac{\sum_{j=1}^d x_j}{m_{\min}},
\]
so, using the boundedness of $\nabla F$ on $S$ and $\PP(A^c)^{1/2} = o(m_{\min}^{-1/2})$, $\EE[\nabla F(\bb{x})^{\top} \bb{\Delta} \, \mathbf \II_A] = o(m_{\min}^{-1})$ uniformly for $\bb{x}\in S$. With $|R_{\bb{m}}^{(A)}(\bb{x})| = o(m_{\min}^{-1})$ and $|T_{\mathrm{far}}| = o(m_{\min}^{-1})$, the expansion
\[
F_{\bb{m}}(\bb{x})-F(\bb{x}) = \frac{1}{2}\sum_{j=1}^d \frac{x_j}{m_j} \, \partial_{x_jx_j}^2F(\bb{x}) + o(m_{\min}^{-1})
\]
follows uniformly for $\bb{x}\in S$. \qed

\subsection{Proof of \texorpdfstring{\hyperref[thm:szasz-bias-var]{Theorem~\ref{thm:szasz-bias-var}}}{Theorem~\ref{thm:szasz-bias-var}}}

\emph{Bias part.} By linearity, $\EE\big[F_{\bb{m},n}(\bb{x})\big] = \EE[F(\bb{K}/\bb{m})] = F_{\bb{m}}(\bb{x})$, so the stated bias expansion follows directly from Proposition~\ref{prop:SM-bias-vector}.

\noindent
\emph{Variance part.} Define the centered summands
\[
Z_{i,\bb{m}}(\bb{x}) := \sum_{\bb{k}\in\NN_0^d}\Big( \II\Big\{\bb{X}_i \leq \frac{\bb{k}}{\bb{m}}\Big\}-F\left(\frac{\bb{k}}{\bb{m}}\right)\Big) \, P_{\bb{k},\bb{m}}(\bb{x}), \qquad i = 1,\dots,n.
\]
Then $F_{\bb{m},n}(\bb{x})-F_{\bb{m}}(\bb{x}) = n^{-1}\sum_{i=1}^n Z_{i,\bb{m}}(\bb{x})$, the $Z_{i,\bb{m}}$ are iid with $\EE [Z_{i,\bb{m}}(\bb{x})] = 0$, and hence
\[
\Var\big(F_{\bb{m},n}(\bb{x})\big) = \frac{1}{n} \, \EE\big[Z_{1,\bb{m}}(\bb{x})^2\big].
\]
A direct expansion yields
\begin{equation}\label{eq:VarZ}
\EE\big[Z_{1,\bb{m}}(\bb{x})^2\big] = \sum_{\bb{k},\bb{\ell}\in\NN_0^d} \bigg(F\Big(\frac{\bb{k}\wedge\bb{\ell}}{\bb{m}}\Big)-F_{\bb{m}}(\bb{x})^2\bigg) P_{\bb{k},\bb{m}}(\bb{x})P_{\bb{\ell},\bb{m}}(\bb{x}),
\end{equation}
where $\bb{k}\wedge\bb{\ell} = (k_1\wedge \ell_1,\dots,k_d\wedge \ell_d)^{\top}$. Equivalently, let $\bb{L}$ be an independent copy of $\bb{K}$,
\[
\EE\big[Z_{1,\bb{m}}(\bb{x})^2\big] = \EE\Big[F\Big(\frac{\bb{K}\wedge\bb{L}}{\bb{m}}\Big)\Big]-\Big(\EE\Big[F\Big(\frac{\bb{K}}{\bb{m}}\Big)\Big]\Big)^2.
\]
Write $\bb{U} := (\bb{K}\wedge\bb{L})/\bb{m}-\bb{x}$ and $\bb{V} := \bb{K}/\bb{m}-\bb{x}$. Then,
\begin{equation}\label{eq:key}
\EE\big[Z_{1,\bb{m}}(\bb{x})^2\big] = \EE\big[F(\bb{x} + \bb{U})\big]-\Big(\EE\big[F(\bb{x} + \bb{V})\big]\Big)^2.
\end{equation}
By the same near/far Taylor argument as in Proposition~\ref{prop:SM-bias-vector}, we obtain, uniformly for $\bb{x}\in S$,
\begin{equation}\label{eq:VarExpansion}
\EE\big[Z_{1,\bb{m}}(\bb{x})^2\big] = F(\bb{x})\{1-F(\bb{x})\} + \sum_{i=1}^d \partial_{x_i}F(\bb{x}) \, \EE [U_i] + O(m_{\min}^{-1}),
\end{equation}
where the $O(m_{\min}^{-1})$ collects the quadratic terms and all remainder bounds. To fully justify \eqref{eq:VarExpansion}, we proceed as follows. Fix the compact set $S$ in the statement and choose $\delta_S > 0$ such that $\mathcal N_{\delta_S}(S)\subseteq (0,\infty)^d$. Let $\eta = \eta_{\bb{m}}\in(0,\delta_S]$ satisfy
\[
\eta_{\bb{m}}\downarrow 0, \qquad \frac{m_{\min}\eta_{\bb{m}}^2}{\log m_{\min}}\to\infty \qquad (m_{\min}\to\infty),
\]
for instance $\eta_{\bb{m}} = \delta_S\wedge m_{\min}^{-1/4}$, and define the two near events
\[
A_{\bb{U}} := \{\|\bb{U}\|_{1} \leq \eta\},\qquad A_{\bb{V}} := \{\|\bb{V}\|_{1} \leq \eta\}.
\]
We first record the elementary moment and tail bounds that will be used below. For $\bb{V}$, $\EE[V_i] = 0$, $\EE[V_i^2] = x_i/m_i = O(m_i^{-1})$, and $\EE[|V_i|^4] = O(m_i^{-2})$, uniformly for $\bb{x}\in S$. For $\bb{U}$, since
\[
|U_i| = \frac{|\min(K_i,L_i)-m_i x_i|}{m_i} \leq \frac{|K_i-m_i x_i|+|L_i-m_i x_i|}{m_i},
\]
the same Poisson moment identities imply $\EE[U_i^2] = O(m_i^{-1})$ and $\EE[|U_i|^4] = O(m_i^{-2})$, uniformly for $\bb{x}\in S$. Moreover, if $0 < a \leq x_j \leq b < \infty$ on $S$, then the union bound and the Poisson tail inequality used in the proof of Proposition~\ref{prop:SM-bias-vector} give
\[
\PP(A_{\bb{V}}^c) \leq 2d \, \exp\left(-c_{\bb{m}}m_{\min}\right),\qquad c_{\bb{m}} := \frac{(\eta_{\bb{m}}/d)^2}{2(b+\eta_{\bb{m}}/d)}.
\]
For $\bb{U}$, the inclusion
\[
\{|U_j|>\eta/d\}\subseteq \left\{\Big|\frac{K_j}{m_j}-x_j\Big|>\eta/d\right\}\cup\left\{\Big|\frac{L_j}{m_j}-x_j\Big|>\eta/d\right\}
\]
yields the same bound, up to a change of the multiplicative constant. Hence, for every fixed $q > 0$,
\[
\PP(A_{\bb{U}}^c)+\PP(A_{\bb{V}}^c) = o(m_{\min}^{-q})
\]
uniformly for $\bb{x}\in S$. In particular, $\PP(A_{\bb{U}}^c)+\PP(A_{\bb{V}}^c) = o(m_{\min}^{-1})$ and $\PP(A_{\bb{U}}^c)^{1/2}+\PP(A_{\bb{V}}^c)^{1/2} = o(m_{\min}^{-1/2})$.

Let $\bb{W}$ denote either $\bb{U}$ or $\bb{V}$, and let $A_{\bb{W}}$ denote the corresponding near event. On $A_{\bb{W}}$, the point $\bb{x}+\bb{W}$ belongs to $\mathcal N_{\delta_S}(S)$. By a second-order Taylor expansion around $\bb{x}$,
\[
F(\bb{x}+\bb{W})-F(\bb{x}) = \sum_{i=1}^d \partial_{x_i}F(\bb{x})W_i + \frac{1}{2}\sum_{i,j=1}^d \partial_{x_i x_j}^2F(\bb{x})W_iW_j + \rho_{\bb{W}},
\]
where, if $\omega_S(\cdot)$ denotes a common modulus of continuity for the entries of the Hessian of $F$ on $\mathcal N_{\delta_S}(S)$,
\[
|\rho_{\bb{W}}| \leq \frac{1}{2} \, \omega_S(\eta)\sum_{i,j=1}^d |W_iW_j| \qquad \text{on } A_{\bb{W}}.
\]
Taking expectations, adding and subtracting the first-order term on $A_{\bb{W}}^c$, and using that $F$ is bounded by one, we obtain
\[
\EE[F(\bb{x}+\bb{W})] = F(\bb{x}) + \sum_{i=1}^d \partial_{x_i}F(\bb{x}) \, \EE[W_i] + R_{\bb{m}}^{(\bb{W})}(\bb{x}),
\]
with
\[
\begin{aligned}
|R_{\bb{m}}^{(\bb{W})}(\bb{x})|
&\leq C\sum_{i=1}^d \EE[|W_i|\II_{A_{\bb{W}}^c}] + C\sum_{i,j=1}^d \EE[|W_iW_j|] + C\omega_S(\eta)\sum_{i,j=1}^d \EE[|W_iW_j|] + \PP(A_{\bb{W}}^c) \\
&= O(m_{\min}^{-1}),
\end{aligned}
\]
uniformly for $\bb{x}\in S$. Indeed, the quadratic terms are $O(m_{\min}^{-1})$ by the moment bounds above, the Taylor remainder is $o(m_{\min}^{-1})$ because $\omega_S(\eta)\to 0$, and the far-event first-order terms satisfy $\EE[|W_i|\II_{A_{\bb{W}}^c}] \leq \EE[W_i^2]^{1/2} \, \PP(A_{\bb{W}}^c)^{1/2} = o(m_{\min}^{-1})$.

Applying this expansion first with $\bb{W}=\bb{U}$ and then with $\bb{W}=\bb{V}$ gives
\[
\EE[F(\bb{x}+\bb{U})] = F(\bb{x}) + \sum_{i=1}^d \partial_{x_i}F(\bb{x}) \, \EE[U_i] + O(m_{\min}^{-1}),
\]
and, since $\EE[V_i] = 0$ for every $i$,
\[
\EE[F(\bb{x}+\bb{V})] = F(\bb{x}) + O(m_{\min}^{-1}).
\]
Consequently,
\[
\Big(\EE[F(\bb{x}+\bb{V})]\Big)^2 = F(\bb{x})^2 + O(m_{\min}^{-1}),
\]
uniformly for $\bb{x}\in S$. Combining the last three displays with \eqref{eq:key} yields \eqref{eq:VarExpansion}, uniformly for $\bb{x}\in S$.

It remains to evaluate $\EE [U_i]$ to first order. The identity $\min(a,b) = \frac{1}{2} \, (a + b-|a-b|)$ gives
\[
\EE [U_i] = \frac{1}{m_i}\Big(\EE[\min(K_i,L_i)]-m_i x_i\Big) = -\frac{1}{2m_i} \, \EE|K_i-L_i|.
\]
Since $K_i-L_i$ has the centered Skellam distribution with variance $2m_i x_i$, its mean absolute value satisfies the classical normal-approximation asymptotic~\citep[see for example][Chapter 5]{BhattacharyaRao2010normalbook}
\[
\EE|K_i-L_i| = \sqrt{\frac{2}{\pi}} \sqrt{2 m_i x_i} \, + \, O\big(m_i^{-1/2}\big), \qquad m_i\to\infty,
\]
uniformly for $x_i$ in compact subsets of $(0,\infty)$. Consequently,
\begin{equation}\label{eq:Ui-expansion}
\EE [U_i] = - \, m_i^{-1/2} \, \sqrt{\frac{x_i}{\pi}} + O(m_i^{-3/2}), \qquad 1 \leq i \leq d,
\end{equation}
uniformly for $\bb{x}\in S$. Insert Equation~\eqref{eq:Ui-expansion} into Equation~\eqref{eq:VarExpansion} to obtain
\[
\EE\big[Z_{1,\bb{m}}(\bb{x})^2\big] = \sigma^2(\bb{x})-\sum_{i=1}^d m_i^{-1/2} \, \partial_{x_i}F(\bb{x}) \, \sqrt{\frac{x_i}{\pi}} + O(m_{\min}^{-1}),
\]
uniformly for $\bb{x}\in S$. Therefore
\[
\Var\big(F_{\bb{m},n}(\bb{x})\big) = \frac{1}{n} \, \sigma^2(\bb{x}) -\frac{1}{n} \, \sum_{i=1}^d m_i^{-1/2} \, \partial_{x_i}F(\bb{x}) \, \sqrt{\frac{x_i}{\pi}} + O\big(n^{-1}m_{\min}^{-1}\big),
\]
uniformly for $\bb{x}\in S$, which is the claimed variance expansion. \qed

\subsection{Proof of \texorpdfstring{\hyperref[thm:deficiency-sm]{Theorem~\ref{thm:deficiency-sm}}}{Theorem~\ref{thm:deficiency-sm}}}

We prove the local result first. The global result follows from the same reciprocal expansion applied to the integrated MSE expansion over $S$. Throughout this proof, $o_{\bb{x}}(\cdot)$ and $O_{\bb{x}}(\cdot)$ denote deterministic orders whose implied constants may depend on $\bb{x}$. Since $0 < F(\bb{x}) < 1$, we have $\sigma^2(\bb{x})>0$. For the empirical cdf,
\begin{equation}\label{eq:edf-mse-thmdef}
\MSE\big(F_k(\bb{x})\big) = \frac{\sigma^2(\bb{x})}{k}, \qquad k\in\NN.
\end{equation}
By Corollary~\ref{cor:szasz-mse} and Theorem~\ref{thm:szasz-bias-var},
\begin{equation}\label{eq:sm-mse-thmdef}
\MSE\big(F_{\bb{m},n}(\bb{x})\big) = \frac{\sigma^2(\bb{x})}{n} -\frac{V(\bb{x})}{n \, m^{1/2}} + \frac{B^2(\bb{x})}{m^2} + r_{n,\bb{x}},
\end{equation}
where
\begin{equation}\label{eq:rn-order-thmdef}
r_{n,\bb{x}} = O_{\bb{x}}\big(n^{-1}m^{-1}\big) + o_{\bb{x}}\big(m^{-2}\big).
\end{equation}
Thus the direct comparison with the empirical cdf is
\[
\MSE\big(F_{\bb{m},n}(\bb{x})\big)-\MSE\big(F_n(\bb{x})\big) = -\frac{V(\bb{x})}{n \, m^{1/2}} + \frac{B^2(\bb{x})}{m^2} + r_{n,\bb{x}}.
\]
By \eqref{eq:edf-mse-thmdef}, the definition of $L^{\mathrm{SM}}(n,\bb{x})$ is equivalently
\begin{equation}\label{eq:ceil-thmdef}
L^{\mathrm{SM}}(n,\bb{x}) = \left\lceil \frac{\sigma^2(\bb{x})}{\MSE\big(F_{\bb{m},n}(\bb{x})\big)} \right\rceil.
\end{equation}
Write
\[
\MSE\big(F_{\bb{m},n}(\bb{x})\big) = \frac{\sigma^2(\bb{x})}{n}\{1+a_{n,\bb{x}}\},
\]
where
\[
a_{n,\bb{x}} = -\frac{V(\bb{x})}{\sigma^2(\bb{x})} \, m^{-1/2} + \frac{B^2(\bb{x})}{\sigma^2(\bb{x})} \, \frac{n}{m^2} + \frac{n r_{n,\bb{x}}}{\sigma^2(\bb{x})}.
\]
If $nm^{-2}\to0$, then $a_{n,\bb{x}}=o_{\bb{x}}(1)$, and therefore \eqref{eq:ceil-thmdef} gives
\[
L^{\mathrm{SM}}(n,\bb{x}) = \left\lceil n(1+a_{n,\bb{x}})^{-1}\right\rceil = n\{1+o_{\bb{x}}(1)\}.
\]
Moreover,
\[
L^{\mathrm{SM}}(n,\bb{x})-n = n\{(1+a_{n,\bb{x}})^{-1}-1\}+O(1) = -n a_{n,\bb{x}}+O_{\bb{x}}(n a_{n,\bb{x}}^2)+O(1).
\]
Using \eqref{eq:rn-order-thmdef},
\[
n^2 r_{n,\bb{x}} = O_{\bb{x}}\Big(\frac{n}{m}\Big) + o_{\bb{x}}\Big(\frac{n^2}{m^2}\Big) = o_{\bb{x}}\Big(\frac{n}{m^{1/2}}\Big) + o_{\bb{x}}\Big(\frac{n^2}{m^2}\Big).
\]
Also, by \eqref{eq:rn-order-thmdef},
\[
a_{n,\bb{x}} = O_{\bb{x}}\big(m^{-1/2}\big) + O_{\bb{x}}\big(nm^{-2}\big) + O_{\bb{x}}\big(m^{-1}\big) + o_{\bb{x}}\big(nm^{-2}\big) = O_{\bb{x}}\big(m^{-1/2}+nm^{-2}\big),
\]
because $m^{-1}=o(m^{-1/2})$ and $nm^{-2}\to0$. Hence
\[
n a_{n,\bb{x}}^2 = O_{\bb{x}}\Big(\frac{n}{m}\Big) + O_{\bb{x}}\Big(\frac{n^2}{m^{5/2}}\Big) + O_{\bb{x}}\Big(\frac{n^3}{m^4}\Big) = o_{\bb{x}}\Big(\frac{n}{m^{1/2}}+\frac{n^2}{m^2}\Big).
\]
Consequently,
\begin{equation}\label{eq:delta-master-thmdef}
L^{\mathrm{SM}}(n,\bb{x})-n = \frac{V(\bb{x})}{\sigma^2(\bb{x})} \, \frac{n}{m^{1/2}} -\frac{B^2(\bb{x})}{\sigma^2(\bb{x})} \, \frac{n^2}{m^2} + o_{\bb{x}}\Big(\frac{n}{m^{1/2}}+\frac{n^2}{m^2}\Big) + O(1).
\end{equation}
If $m n^{-2/3}\to\infty$ and $m n^{-2}\to0$, then $n/m^{1/2}\to\infty$ and $n^2/m^2=o(n/m^{1/2})$. Hence \eqref{eq:delta-master-thmdef} yields
\[
L^{\mathrm{SM}}(n,\bb{x})-n = \frac{n}{m^{1/2}}\left\{\frac{V(\bb{x})}{\sigma^2(\bb{x})}+o_{\bb{x}}(1)\right\}.
\]
If $m n^{-2/3}\to c\in(0,\infty)$, then
\[
\frac{n}{m^{1/2}} = c^{-1/2}n^{2/3}\{1+o(1)\}, \qquad \frac{n^2}{m^2} = c^{-2}n^{2/3}\{1+o(1)\},
\]
and \eqref{eq:delta-master-thmdef} gives
\[
L^{\mathrm{SM}}(n,\bb{x})-n = n^{2/3} \left\{ c^{-1/2}\frac{V(\bb{x})}{\sigma^2(\bb{x})} - c^{-2}\frac{B^2(\bb{x})}{\sigma^2(\bb{x})} + o_{\bb{x}}(1) \right\}.
\]

For the global result, let
\[
\Sigma_S := \int_S\sigma^2(\bb{x}) \, \mathrm{d}\bb{x}, \qquad \mathcal V_S := \int_S V(\bb{x}) \, \mathrm{d}\bb{x}, \qquad \mathcal B_S := \int_S B^2(\bb{x}) \, \mathrm{d}\bb{x}.
\]
Since $\Sigma_S>0$,
\[
\IMSE_S(F_k) = \frac{\Sigma_S}{k}, \qquad k\in\NN,
\]
and Corollary~\ref{thm:szasz-mise}, together with Theorem~\ref{thm:szasz-bias-var}, gives
\[
\IMSE_S(F_{\bb{m},n}) = \frac{\Sigma_S}{n} -\frac{\mathcal V_S}{n \, m^{1/2}} + \frac{\mathcal B_S}{m^2} + O\big(n^{-1}m^{-1}\big) + o\big(m^{-2}\big).
\]
Repeating the preceding reciprocal expansion with $\sigma^2(\bb{x})$, $V(\bb{x})$, and $B^2(\bb{x})$ replaced by $\Sigma_S$, $\mathcal V_S$, and $\mathcal B_S$ proves
\[
G^{\mathrm{SM}}_S(n) = n\{1+o(1)\}
\]
whenever $nm^{-2}\to0$, and gives the two displayed global expansions in \eqref{eq:def-sm-regime-a} and \eqref{eq:def-sm-regime-b}. The divergence conclusions follow immediately from the corresponding positivity conditions stated in the theorem. \qed

\subsection{Proof of \texorpdfstring{\hyperref[thm:sm-cdf-clt-interior]{Theorem~\ref{thm:sm-cdf-clt-interior}}}{Theorem~\ref{thm:sm-cdf-clt-interior}}}

Fix $\bb{x}\in (0,\infty)^d$ such that $0 < F(\bb{x}) < 1$. For each $n$ and multi-index $\bb{m} = \bb{m}(n)$, define
\[
Y_{i,\bb{m}}(\bb{x}) := \prod_{j=1}^d \PP\left(\mathrm{Poi}(m_j x_j) \geq W_{ij}\right), \qquad 1 \leq i \leq n,
\]
where $W_{ij} = \lceil m_j X_{ij}\rceil$, so that
\[
F_{\bb{m},n}(\bb{x}) = \frac{1}{n}\sum_{i=1}^n Y_{i,\bb{m}}(\bb{x}).
\]
By construction, for each fixed $(n,\bb{m})$, the random variables $Y_{1,\bb{m}}(\bb{x}),\ldots,Y_{n,\bb{m}}(\bb{x})$ are iid, and
\[
0 \leq Y_{i,\bb{m}}(\bb{x}) \leq 1 \qquad\text{a.s.},\qquad 1 \leq i \leq n.
\]
Write
\[
\EE\big[Y_{1,\bb{m}}(\bb{x})\big] = \EE\big[F_{\bb{m},n}(\bb{x})\big],
\]
and
\[
\Var\big(F_{\bb{m},n}(\bb{x})\big) = \Var\Big(\frac1n\sum_{i=1}^n Y_{i,\bb{m}}(\bb{x})\Big) = \frac{\Var\left(Y_{1, \bb{m}}(\bb{x})\right)}{n}.
\]
From the interior bias--variance expansion in Theorem~\ref{thm:szasz-bias-var}, we obtain
\[
\Var\left(Y_{1, \bb{m}}(\bb{x})\right) = n \, \Var\big(F_{\bb{m},n}(\bb{x})\big) = F(\bb{x})\{1-F(\bb{x})\} + O\big(m_{\min}^{-1/2}\big).
\]
Hence
\begin{equation*}
\Var\left(Y_{1, \bb{m}}(\bb{x})\right) \to F(\bb{x})\{1-F(\bb{x})\} > 0 \qquad\text{as } n\to\infty.
\end{equation*}

\paragraph{(a) CLT around the mean.} Define the standardized sum
\[
Z_{n,\bb{m}}(\bb{x}) := \frac{1}{\sqrt{n \, \Var\left(Y_{1, \bb{m}}(\bb{x})\right)}} \sum_{i=1}^n \big\{Y_{i,\bb{m}}(\bb{x})- \EE\big[Y_{1,\bb{m}}(\bb{x})\big]\big\}.
\]
Then
\[
Z_{n,\bb{m}}(\bb{x}) = \frac{\sqrt{n}}{ \sqrt{\Var\left(Y_{1, \bb{m}}(\bb{x})\right)}} \Big\{F_{\bb{m},n}(\bb{x})- \EE\big[Y_{1,\bb{m}}(\bb{x})\big]\Big\}.
\]
We use the Lindeberg--Feller CLT for the triangular array $\{Y_{i,\bb{m}}(\bb{x}):1 \leq i \leq n, n\in \mathbb{N}\}$. For any $\varepsilon > 0$,
\[
\begin{aligned}
&\frac{1}{n \, \Var\left(Y_{1, \bb{m}}(\bb{x})\right)} \sum_{i=1}^n \EE\Big[ \big(Y_{i,\bb{m}}(\bb{x})-\EE\big[Y_{1,\bb{m}}(\bb{x})\big]\big)^2 \II\!\big\{ \big|Y_{i,\bb{m}}(\bb{x})-\EE\big[Y_{1,\bb{m}}(\bb{x})\big]\big| > \varepsilon\sqrt{n \, \Var\left(Y_{1, \bb{m}}(\bb{x})\right)} \big\} \Big] \\
&\leq \frac{1}{n \, \Var\left(Y_{1, \bb{m}}(\bb{x})\right)} \sum_{i=1}^n \EE\Big[ \big(Y_{i,\bb{m}}(\bb{x})-\EE\big[Y_{1,\bb{m}}(\bb{x})\big]\big)^2 \II\!\big\{ 1 > \varepsilon\sqrt{n \, \Var\left(Y_{1, \bb{m}}(\bb{x})\right)} \big\} \Big].
\end{aligned}
\]
Since $|Y_{i,\bb{m}}(\bb{x})-\EE[Y_{1,\bb{m}}(\bb{x})]| \leq 1$ a.s.\ and $\Var\left(Y_{1, \bb{m}}(\bb{x})\right)\to F(\bb{x})\{1-F(\bb{x})\} > 0$, we have
\[
\varepsilon\sqrt{n \, \Var\left(Y_{1, \bb{m}}(\bb{x})\right)}\to\infty, \qquad \text{as } n\to\infty.
\]
Thus, for $n$ large enough, the indicator is identically zero, so the Lindeberg--Feller condition is satisfied and
\[
Z_{n,\bb{m}}(\bb{x})\xrightarrow{d}\mathcal N(0,1).
\]
We obtain
\[
\sqrt{n} \, \Big\{F_{\bb{m},n}(\bb{x})-\EE\big[Y_{1,\bb{m}}(\bb{x})\big]\Big\} = \sqrt{\Var\left(Y_{1, \bb{m}}(\bb{x})\right)} \, Z_{n,\bb{m}}(\bb{x}) \xrightarrow{d} \mathcal N\Big(0,F(\bb{x})\{1-F(\bb{x})\}\Big),
\]
which proves Equation~\eqref{eq:sm-clt-interior-mean}.

\paragraph{(b) CLT around $F(\bb{x})$.} We write
\[
\sqrt{n} \, \big\{F_{\bb{m},n}(\bb{x})-F(\bb{x})\big\} = \sqrt{n} \, \big\{F_{\bb{m},n}(\bb{x})-\EE\big[Y_{1,\bb{m}}(\bb{x})\big]\big\} + \sqrt{n} \, \big\{\EE\big[Y_{1,\bb{m}}(\bb{x})\big]-F(\bb{x})\big\}.
\]
The first term converges in distribution to $\mathcal N(0,F(\bb{x})\{1-F(\bb{x})\})$ by part (a). For the second term, use the bias bound in Theorem~\ref{thm:szasz-bias-var}, i.e.,
\[
\big|\EE\big[Y_{1,\bb{m}}(\bb{x})\big]-F(\bb{x})\big| = \big|\Bias\big(F_{\bb{m},n}(\bb{x})\big)\big| = O\big(m_{\min}^{-1}\big).
\]
Therefore, under the bandwidth condition ${\sqrt{n}}m^{-1}_{\min}\to 0$,
\[
\sqrt{n} \, \big|\EE\big[Y_{1,\bb{m}}(\bb{x})\big]-F(\bb{x})\big| = O\left(\sqrt{n}m^{-1}_{\min}\right) \to 0.
\]
Hence the second term converges to $0$ in probability. By Slutsky's theorem,
\[
\sqrt{n} \, \big\{F_{\bb{m},n}(\bb{x})-F(\bb{x})\big\} \xrightarrow{d} \mathcal N\Big(0,F(\bb{x})\{1-F(\bb{x})\}\Big),
\]
which completes the proof. \qed

\subsection{Proof of \texorpdfstring{\hyperref[thm:sm-cdf-uniform]{Theorem~\ref{thm:sm-cdf-uniform}}}{Theorem~\ref{thm:sm-cdf-uniform}}}

Fix $L > 0$ and $\bb{x}\in[0,L]^d$. Recall that
\[
F_{\bb{m},n}(\bb{x}) = \sum_{\bb{k}\in\NN_0^d}F_n\Big(\frac{\bb{k}}{\bb{m}}\Big) \, P_{\bb{k},\bb{m}}(\bb{x}), \qquad F_{\bb{m}}(\bb{x}) = \sum_{\bb{k}\in\NN_0^d}F\Big(\frac{\bb{k}}{\bb{m}}\Big) \, P_{\bb{k},\bb{m}}(\bb{x}),
\]
where $P_{\bb{k},\bb{m}}(\bb{x}) \geq 0$ and $\sum_{\bb{k}\in\NN_0^d}P_{\bb{k},\bb{m}}(\bb{x}) = 1$. Hence,
\[
F_{\bb{m},n}(\bb{x})-F_{\bb{m}}(\bb{x}) = \sum_{\bb{k}\in\NN_0^d}\Big\{F_n\Big(\frac{\bb{k}}{\bb{m}}\Big)-F\Big(\frac{\bb{k}}{\bb{m}}\Big)\Big\} \, P_{\bb{k},\bb{m}}(\bb{x}),
\]
and by convexity of the weighted average,
\begin{equation}\label{eq:sm-stoch-gc}
\big|F_{\bb{m},n}(\bb{x})-F_{\bb{m}}(\bb{x})\big| \leq \sup_{\bb{t}\in[0,\infty)^d}\big|F_n(\bb{t})-F(\bb{t})\big|.
\end{equation}
Taking the supremum over $\bb{x}\in[0,L]^d$ gives
\[
\sup_{\bb{x}\in[0,L]^d}\big|F_{\bb{m},n}(\bb{x})-F_{\bb{m}}(\bb{x})\big| \leq \sup_{\bb{t}\in[0,\infty)^d}\big|F_n(\bb{t})-F(\bb{t})\big| \xrightarrow{a.s.}0, \qquad n\to\infty,
\]
by the multivariate Glivenko--Cantelli theorem \citep[see, e.g.,][Chapter~19]{Van_der_vaart1996}.

Further, write
\[
F_{\bb{m}}(\bb{x}) = \EE\Big[F\Big(\frac{\bb{K}}{\bb{m}}\Big)\Big],
\]
where $K_1,\ldots,K_d$ are independent with $K_j\sim\mathrm{Poi}(m_jx_j)$, and set $\bb{\Delta} := \bb{K}/\bb{m}-\bb{x}$. Fix $\delta\in(0,1]$. Since $F$ is continuous on every compact, it is uniformly continuous on $[0,L + 1]^d$; let $\omega_{L + 1}(\delta)$ denote its modulus of continuity on $[0,L + 1]^d$. For $\bb{x}\in[0,L]^d$, the event $\{\|\bb{\Delta}\| \leq \delta\}$ implies $\bb{K}/\bb{m}\in[0,L + 1]^d$ (because $\bb{K}/\bb{m} \geq 0$ componentwise), hence
\[
\big|F_{\bb{m}}(\bb{x})-F(\bb{x})\big| = \Big|\EE\Big[F\Big(\frac{\bb{K}}{\bb{m}}\Big)-F(\bb{x})\Big]\Big| \leq \omega_{L + 1}(\delta) + \PP(\|\bb{\Delta}\| > \delta).
\]
Moreover,
\[
\EE\|\bb{\Delta}\|^2 = \sum_{j=1}^d \Var\Big(\frac{K_j}{m_j}\Big) = \sum_{j=1}^d \frac{x_j}{m_j} \leq \frac{dL}{m_{\min}},
\]
so Markov's inequality yields $\PP(\|\bb{\Delta}\| > \delta) \leq (dL)/(\delta^2 m_{\min})$, uniformly over $\bb{x}\in[0,L]^d$. Therefore,
\[
\sup_{\bb{x}\in[0,L]^d}\big|F_{\bb{m}}(\bb{x})-F(\bb{x})\big| \leq \omega_{L + 1}(\delta) + \frac{dL}{\delta^2 m_{\min}} \to 0, \qquad m_{\min}\to\infty,
\]
by first letting $\delta\downarrow 0$ (so $\omega_{L + 1}(\delta)\downarrow 0$) and then letting $m_{\min}\to\infty$.

Combining the stochastic bound \eqref{eq:sm-stoch-gc} with the bias bound above gives
\[
\sup_{\bb{x}\in[0,L]^d}\big|F_{\bb{m},n}(\bb{x})-F(\bb{x})\big| \leq \sup_{\bb{t}\in[0,\infty)^d}\big|F_n(\bb{t})-F(\bb{t})\big| + \sup_{\bb{x}\in[0,L]^d}\big|F_{\bb{m}}(\bb{x})-F(\bb{x})\big| \xrightarrow{a.s.}0,
\]
as $n\to\infty$ and $m_{\min}\to\infty$. This proves the theorem. \qed

\subsection{Proof of \texorpdfstring{\hyperref[thm:boundary-bias]{Theorem~\ref{thm:boundary-bias}}}{Theorem~\ref{thm:boundary-bias}}}

Throughout, $(\lambda_1,\dots,\lambda_d)\in[0,L]^d$ is fixed and
\[
\bb{x} = \Big(\frac{\lambda_1}{m_1},\dots,\frac{\lambda_d}{m_d}\Big), \qquad m_{\min} = \min_{1 \leq j \leq d} m_j \to \infty.
\]

We start by proving \eqref{thm:bias.boundary}. As in Theorem~\ref{thm:szasz-bias-var}, let $K_1,\dots,K_d$ be independent with $K_j \sim \mathrm{Poi}(m_j x_j) = \mathrm{Poi}(\lambda_j)$ and define $Z_j := K_j/m_j$. Then
\[
\EE\big[F_{\bb{m},n}(\bb{x})\big] = \EE\Big[F\Big(\frac{\bb{K}}{\bb{m}}\Big)\Big] = F_{\bb{m}}(\bb{x}), \qquad \Bias\big(F_{\bb{m},n}(\bb{x})\big) = F_{\bb{m}}(\bb{x}) - F(\bb{x}),
\]
so it suffices to expand $F_{\bb{m}}(\bb{x}) - F(\bb{x})$.

For each $j\in\{1,\dots,d\}$, define the univariate slice
\[
\varphi_j(t) := F(x_1,\dots,x_{j-1},t,x_{j + 1},\dots,x_d), \qquad t \geq 0.
\]
By Assumption~\ref{ass:szasz-C2-boundary}, there exists $\delta > 0$ such that $F$ is twice continuously differentiable on $\{x: \min_k x_k \leq \delta\}$ and the mixed second derivatives are bounded there. In particular, for each $j$, $\varphi_j$ is twice continuously differentiable on $[0,\delta]$ with
\[
\sup_{t\in[0,\delta]} |\varphi_j''(t)| < \infty,
\]
uniformly in $(\lambda_1,\dots,\lambda_d)\in[0,L]^d$. A second-order Taylor expansion at the boundary point $t = 0$ yields
\[
\varphi_j(t) = \varphi_j(0) + t \, \varphi_j'(0) + \frac{1}{2} t^2 \varphi_j''(0) + R_j(t), \qquad t\in[0,\delta],
\]
where the remainder can be chosen in the standard form
\begin{equation}\label{eq:Rj-modulus}
|R_j(t)| \leq \omega_j(|t|) \, t^2,\qquad t\in[0,\delta],
\end{equation}
for some function $\omega_j:[0,\infty)\to[0,\infty)$ with $\omega_j(r)\to0$ as $r\downarrow0$, independent of $(\lambda_1,\dots,\lambda_d)$.

Evaluating at $t = Z_j$ and using
\[
\EE[Z_j] = x_j = \frac{\lambda_j}{m_j}, \qquad \EE[Z_j^2] = x_j^2 + \frac{\lambda_j}{m_j^2},
\]
we obtain
\begin{equation}\label{eq:phiZ-1}
\EE[\varphi_j(Z_j)] = \varphi_j(0) + x_j\varphi_j'(0) + \frac{1}{2}\Big(x_j^2 + \frac{\lambda_j}{m_j^2}\Big)\varphi_j''(0) + \EE[R_j(Z_j)].
\end{equation}

Since $x_j = \lambda_j/m_j \to 0$ uniformly in $\lambda_j\in[0,L]$, we may also expand $\varphi_j$ at $t = x_j$:
\[
\varphi_j(x_j) = \varphi_j(0) + x_j\varphi_j'(0) + \frac{1}{2} x_j^2 \varphi_j''(0) + r_j,
\]
where $r_j = o(x_j^2) = o(m_j^{-2})$ uniformly in $(\lambda_1,\dots,\lambda_d)$. Subtracting this from Equation~\eqref{eq:phiZ-1} gives
\begin{equation}\label{eq:slice-diff}
\EE[\varphi_j(Z_j)] - \varphi_j(x_j) = \frac{1}{2} \, \frac{\lambda_j}{m_j^2} \, \varphi_j''(0) + \big(\EE[R_j(Z_j)] - r_j\big).
\end{equation}
To identify the leading term we need the behaviour of $\EE[R_j(Z_j)]$.

Fix $\varepsilon > 0$ and choose $\rho\in(0,\delta]$ such that $\omega_j(r) \leq \varepsilon$ whenever $0 \leq r \leq \rho$ in Equation~\eqref{eq:Rj-modulus}. Write
\[
\EE[R_j(Z_j)] = \EE\big[R_j(Z_j)\II_A\big] + \EE\big[R_j(Z_j)\II_{A^c}\big], \qquad A := \{Z_j \leq \rho\}.
\]
On $A$ we have $|R_j(Z_j)| \leq \varepsilon Z_j^2$ by Equation~\eqref{eq:Rj-modulus}, hence
\[
\big|\EE[R_j(Z_j)\II_A]\big| \leq \varepsilon \, \EE[Z_j^2] = \varepsilon \, \Big(x_j^2 + \frac{\lambda_j}{m_j^2}\Big) = O(\varepsilon m_j^{-2}),
\]
uniformly in $(\lambda_1,\dots,\lambda_d)\in[0,L]^d$.

On $A^c = \{Z_j > \rho\}$ we cannot guarantee $Z_j \leq \delta$, so we use a global bound for the Taylor remainder. Since $0 \leq F \leq 1$ and Assumption~\ref{ass:szasz-C2-boundary} implies that $\varphi_j'(0)$ and $\varphi_j''(0)$ are uniformly bounded in $\lambda$, there exists a constant $C_j < \infty$ such that
\[
|R_j(t)| = \bigl|\varphi_j(t) - \varphi_j(0) - t\varphi_j'(0) - \frac{1}{2} t^2\varphi_j''(0)\bigr| \leq C_j(1 + t^2), \qquad t \geq 0,
\]
uniformly in $\lambda$. Hence, on $A^c$,
\[
|R_j(Z_j)| \leq C_j\bigl(1 + Z_j^2\bigr) = C_j\Bigl(1 + \frac{K_j^2}{m_j^2}\Bigr),
\]
and therefore
\begin{align*}
\bigl|\EE\bigl[R_j(Z_j)\II_{A^c}\bigr]\bigr| &\leq C_j \, \PP(K_j > \rho m_j) + \frac{C_j}{m_j^2}\EE\bigl[K_j^2 \II_{\{K_j > \rho m_j\}}\bigr].
\end{align*}
For any fixed $\rho > 0$ and $L < \infty$, a Chernoff bound for Poisson tails yields constants $C,c > 0$ (depending only on $\rho$ and $L$) such that
\[
\sup_{\lambda_j\in[0,L]} \PP(K_j > \rho m_j) \leq C e^{-c m_j} = o(m_j^{-2}) \qquad (m_j\to\infty).
\]
Moreover, $\sup_{\lambda_j\in[0,L]} \EE[K_j^4] < \infty$, so the Cauchy--Schwarz inequality yields
\[
\frac{1}{m_j^2} \sup_{\lambda_j\in[0,L]} \EE\bigl[K_j^2 \, \II_{\{K_j > \rho m_j\}}\bigr] \leq \frac{1}{m_j^2} \sqrt{\sup_{\lambda_j\in[0,L]} \EE[K_j^4]} \times \sqrt{\sup_{\lambda_j\in[0,L]} \PP(K_j > \rho m_j)} = o(m_j^{-2}).
\]
We conclude that
\[
\EE\bigl[R_j(Z_j) \, \II_{A^c}\bigr] = o(m_j^{-2}),
\]
uniformly in $\lambda_j\in[0,L]$. Combining the two parts,
\[
\EE[R_j(Z_j)] = O(\varepsilon m_j^{-2}) + o(m_j^{-2}),
\]
uniformly in $(\lambda_1,\dots,\lambda_d)$; since $\varepsilon > 0$ is arbitrary,
\[
\EE[R_j(Z_j)] = o(m_j^{-2}) \qquad\text{uniformly in $(\lambda_1,\dots,\lambda_d)\in[0,L]^d$}.
\]
Recalling $r_j = o(m_j^{-2})$, we deduce from Equation~\eqref{eq:slice-diff} that
\[
\EE[\varphi_j(Z_j)] - \varphi_j(x_j) = \frac{1}{2} \, \frac{\lambda_j}{m_j^2} \, \varphi_j''(0) + o(m_j^{-2}).
\]
Since
\[
\varphi_j''(0) = \partial^2_{x_j x_j} F\big(x^{(j,0)}\big),
\]
we arrive at the slice expansion
\begin{equation}\label{eq:slice-final}
\EE[\varphi_j(Z_j)] - \varphi_j(x_j) = \frac{1}{2} \, \frac{\lambda_j}{m_j^2} \, \partial^2_{x_j x_j} F\big(x^{(j,0)}\big) + o(m_j^{-2}), \qquad j = 1,\dots,d,
\end{equation}
with the $o(m_j^{-2})$ term uniform in $(\lambda_1,\dots,\lambda_d)\in[0,L]^d$.

Write $Z = (Z_1,\dots,Z_d)$ and define intermediate vectors
\[
Z^{(0)} := x,\qquad Z^{(j)} := (Z_1,\dots,Z_j,x_{j + 1},\dots,x_d),\qquad j = 1,\dots,d.
\]
Then $Z^{(d)} = Z$ and
\[
F(Z) - F(x) = \sum_{j=1}^d \big(F(Z^{(j)}) - F(Z^{(j-1)})\big).
\]
Taking expectations,
\begin{equation}\label{eq:telescoping}
F_m(x) - F(x) = \EE[F(Z)] - F(x) = \sum_{j=1}^d \EE\big[F(Z^{(j)}) - F(Z^{(j-1)})\big].
\end{equation}
For each $j$, add and subtract the slice $\varphi_j$ at $Z_j$ and $x_j$:
\[
F(Z^{(j)}) - F(Z^{(j-1)}) = \big(F(Z^{(j)}) - \varphi_j(Z_j)\big) + \big(\varphi_j(Z_j) - \varphi_j(x_j)\big) + \big(\varphi_j(x_j) - F(Z^{(j-1)})\big).
\]
Set
\[
A_{j,1} := \EE\big[F(Z^{(j)}) - \varphi_j(Z_j)\big],\qquad A_{j,2} := \EE[\varphi_j(Z_j)] - \varphi_j(x_j),\qquad A_{j,3} := \varphi_j(x_j) - \EE[F(Z^{(j-1)})].
\]
Then
\[
\EE\big[F(Z^{(j)}) - F(Z^{(j-1)})\big] = A_{j,1} + A_{j,2} + A_{j,3}.
\]

By definition,
\[
\varphi_j(Z_j) = F(x_1,\dots,x_{j-1},Z_j,x_{j + 1},\dots,x_d),
\]
whereas
\[
F(Z^{(j)}) = F(Z_1,\dots,Z_{j-1},Z_j,x_{j + 1},\dots,x_d).
\]
Thus $A_{j,1}$ measures the effect of replacing $(x_1,\dots,x_{j-1})$ by $(Z_1,\dots,Z_{j-1})$ in the first $j-1$ coordinates (with the $j$th coordinate fixed at $Z_j$). Similarly, $A_{j,3}$ measures the effect of replacing $(Z_1,\dots,Z_{j-1})$ by $(x_1,\dots,x_{j-1})$ when the $j$th coordinate is fixed at $x_j$.

Consider $A_{j,1}$. For each realization of $(Z_1,\dots,Z_{j-1},Z_j)$ with all components in $[0,\delta]$, let
\[
g(y_1,\ldots,y_{j-1}) := F(y_1,\ldots,y_{j-1},Z_j,x_{j + 1},\ldots,x_d),
\]
and set $\bb{h} := (Z_1-x_1,\ldots,Z_{j-1}-x_{j-1})^{\top}$. By the integral form of Taylor's theorem,
\begin{align*}
F(Z^{(j)})-\varphi_j(Z_j)
&= g(x_{1:(j-1)} + \bb{h})-g(x_{1:(j-1)}) \\
&= \sum_{k=1}^{j-1}\partial_{x_k}g(x_{1:(j-1)}) \, h_k + \int_0^1 (1-t) \, \bb{h}^{\top} \nabla^2 g(x_{1:(j-1)} + t \bb{h}) \, \bb{h}\mathrm{d} t.
\end{align*}
Note that $\partial_{x_k}g(x_{1:(j-1)}) = \partial_{x_k}F(x_1,\ldots,x_{j-1},Z_j,x_{j + 1},\ldots,x_d)$ depends only on $Z_j$, whereas $h_k = Z_k-x_k$ depends only on $Z_k$. Since the components $(Z_1,\ldots,Z_{j-1})$ are independent of $Z_j$ and $\EE[Z_k-x_k] = 0$, we obtain
\[
\EE\left[\partial_{x_k}g(x_{1:(j-1)})(Z_k-x_k)\right] = \EE\left[\partial_{x_k}g(x_{1:(j-1)})\right]\EE[Z_k-x_k] = 0, \qquad k = 1,\ldots,j-1.
\]
Moreover, by Assumption~\ref{ass:szasz-C2-boundary}, the second-order partial derivatives of $F$ are uniformly bounded on $\{x:\min_{1 \leq j \leq d}x_j \leq \delta\}$; hence there exists $C < \infty$ such that $\|\nabla^2 g(\cdot)\|_{\infty} \leq C$ throughout the segment $\{x_{1:(j-1)} + t \bb{h}: t\in[0,1]\}$. Consequently,
\[
\left| \int_0^1 (1-t) \, \bb{h}^{\top} \nabla^2 g(x_{1:(j-1)} + t \bb{h}) \, \bb{h} \mathrm{d} t \right| \leq \frac{C}{2} \, \|\bb{h}\|_1^2 \leq \frac{C (j-1)}{2} \sum_{k=1}^{j-1}|Z_k-x_k|^2,
\]
and
\[
\EE[|Z_k-x_k|^2] = \frac{\lambda_k}{m_k^2},
\]
so we obtain
\[
|A_{j,1}| \leq \frac{C (j-1)}{2} \sum_{k=1}^{j-1} \EE\big[|Z_k-x_k|^2\big] \leq C' (j-1) \sum_{k=1}^{j-1} \frac{\lambda_k}{m_k^2} \leq C'' \, m_{\min}^{-2},
\]
uniformly in $(\lambda_1,\dots,\lambda_d)\in[0,L]^d$. An analogous argument, with the $j$th coordinate fixed at $x_j$ instead of $Z_j$, gives the same bound for $A_{j,3}$. In both cases one can refine $O(m_{\min}^{-2})$ to $o(m_{\min}^{-2})$ using continuity of the derivatives, but for the global expansion it suffices that
\[
A_{j,1} = o(m_{\min}^{-2}),\qquad A_{j,3} = o(m_{\min}^{-2}),
\]
uniformly in $(\lambda_1,\dots,\lambda_d)$.

From Equation~\eqref{eq:slice-final},
\[
A_{j,2} = \EE[\varphi_j(Z_j)] - \varphi_j(x_j) = \frac{1}{2} \, \frac{\lambda_j}{m_j^2} \, \partial^2_{x_j x_j} F\big(x^{(j,0)}\big) + o(m_j^{-2}).
\]
Combining with the bounds for $A_{j,1}$ and $A_{j,3}$, we obtain
\[
\EE\big[F(Z^{(j)}) - F(Z^{(j-1)})\big] = \frac{1}{2} \, \frac{\lambda_j}{m_j^2} \, \partial^2_{x_j x_j} F\big(x^{(j,0)}\big) + o(m_{\min}^{-2}), \qquad j = 1,\dots,d,
\]
uniformly in $(\lambda_1,\dots,\lambda_d)\in[0,L]^d$. Substituting into Equation~\eqref{eq:telescoping} and summing over $j$ yields
\[
F_m(x) - F(x) = \frac{1}{2} \sum_{j=1}^d \frac{\lambda_j}{m_j^2} \, \partial^2_{x_j x_j} F\big(x^{(j,0)}\big) + o(m_{\min}^{-2}),
\]
uniformly in $(\lambda_1,\dots,\lambda_d)\in[0,L]^d$. Since $\Bias(F_{m,n}(x)) = F_m(x)-F(x)$, this is precisely the claimed boundary-layer bias expansion in \eqref{thm:bias.boundary}.

Next, we prove \eqref{thm:var.boundary}. Write
\[
F_{\bb{m},n}(\bb{x})-F_{\bb{m}}(\bb{x}) = \frac{1}{n}\sum_{i=1}^n Z_{i,\bb{m}}(\bb{x}),
\]
where
\[
Z_{i,\bb{m}}(\bb{x}) := \sum_{k\in\mathbb{N}_0^d} \Bigl\{\II_{\{\bb{X}_i \leq \bb{k}/\bb{m}\}}-F(\bb{k}/\bb{m})\Bigr\}P_{\bb{k},\bb{m}}(\bb{x}), \qquad i = 1,\dots,n,
\]
so that the $Z_{i,\bb{m}}(\bb{x})$ are iid with $\EE[Z_{i,\bb{m}}(\bb{x})] = 0$ and
\[
\Var\big(F_{\bb{m},n}(\bb{x})\big) = \frac{1}{n} \, \EE\bigl[Z_{1,\bb{m}}(\bb{x})^2\bigr].
\]
As in Equation~\eqref{eq:VarZ} in the proof of Theorem~\ref{thm:szasz-bias-var}, let $\bb{K} = (K_1,\dots,K_d)$ and $\bb{L} = (L_1,\dots,L_d)$ be independent with
\[
K_j\sim\mathrm{Poi}(m_j x_j) = \mathrm{Poi}(\lambda_j),\qquad L_j\sim\mathrm{Poi}(m_j x_j) = \mathrm{Poi}(\lambda_j),\qquad j = 1,\dots,d,
\]
and define
\[
\bb{U} := \frac{\bb{K}\wedge \bb{L}}{\bb{m}}-\bb{x},\qquad \bb{V} := \frac{\bb{K}}{\bb{m}}-\bb{x}.
\]
Then one can write
\[
\EE\bigl[Z_{1,\bb{m}}(\bb{x})^2\bigr] = F(\bb{x})\bigl\{1-F(\bb{x})\bigr\} + \sum_{j=1}^d \partial_{x_j}F(\bb{x}) \, \EE[U_j] + R_{\bb{m}}(\bb{x}),
\]
where $R_{\bb{m}}(\bb{x})$ collects the quadratic and Taylor remainder terms. The derivation of this representation is based on a near/far Taylor expansion of $F(\bb{x} + \bb{U})$ and $F(\bb{x} + \bb{V})$ and only requires that $F$ be twice continuously differentiable with uniformly bounded second derivatives on a neighbourhood of $\bb{x}$, together with Poisson tail bounds; see the detailed argument in the proof of Theorem~\ref{thm:szasz-bias-var}. The only change here is that $\bb{x} = \bb{x}( \bb{m})$ depends on $\bb{m}$, but for $x_j = \lambda_j/m_j$ and $(\lambda_1,\dots,\lambda_d)\in[0,L]^d$ we have $\min_j x_j \leq \delta$ for all sufficiently large $m_{\min}$, so $\bb{x}$ remains in the region covered by Assumption~\ref{ass:szasz-C2-boundary}. Under Assumption~\ref{ass:szasz-C2-boundary}, this yields
\[
R_{\bb{m}}(\bb{x}) = O(m_{\min}^{-1}),
\]
uniformly in $(\lambda_1,\dots,\lambda_d)\in[0,L]^d$.

It remains to control $\EE[U_j]$ under the boundary-layer scaling. By the identity $\min(a,b) = \frac{1}{2}(a + b-|a-b|)$ we have
\[
\EE[U_j] = \frac{1}{m_j}\Bigl\{\EE\bigl[\min(K_j,L_j)\bigr]-m_jx_j\Bigr\} = \frac{1}{m_j}\Bigl\{\EE\bigl[\min(K_j,L_j)\bigr]-\lambda_j\Bigr\}.
\]
Here $K_j,L_j\sim\mathrm{Poi}(\lambda_j)$ with $\lambda_j\in[0,L]$ fixed, so $\EE[\min(K_j,L_j)]$ is a bounded function of $\lambda_j$. In particular,
\[
0 \leq \EE\bigl[\min(K_j,L_j)\bigr] \leq \EE[K_j] = \lambda_j
\]
implies
\[
-\lambda_j \leq \EE\bigl[\min(K_j,L_j)\bigr]-\lambda_j \leq 0,
\]
and therefore
\[
|\EE[U_j]| \leq \frac{\lambda_j}{m_j} \leq \frac{L}{m_j}.
\]
Consequently,
\[
\Bigl|\sum_{j=1}^d \partial_{x_j}F(x) \, \EE[U_j]\Bigr| \leq C \, m_{\min}^{-1}
\]
for some constant $C < \infty$ depending only on $L$, $d$ and the derivative bounds from Assumption~\ref{ass:szasz-C2-boundary}.

Putting the pieces together, we obtain
\[
\EE\bigl[Z_{1,\bb{m}}(\bb{x})^2\bigr] = F(\bb{x})\bigl\{1-F(\bb{x})\bigr\} + O(m_{\min}^{-1}),
\]
uniformly over $(\lambda_1,\dots,\lambda_d)\in[0,L]^d$, as $m_{\min}\to\infty$. Therefore
\[
\Var\big(F_{\bb{m},n}(\bb{x})\big) = \frac{1}{n} \, \EE\bigl[Z_{1,\bb{m}}(\bb{x})^2\bigr] = \frac{1}{n} \, F(\bb{x})\bigl\{1-F(\bb{x})\bigr\} + O\bigl(n^{-1}m_{\min}^{-1}\bigr),
\]
which is the claimed variance expansion in \eqref{thm:bias.boundary}. This concludes the proof.

\end{appendices}

\section*{Reproducibility}

The $\mathsf{R}$~\citep{Rsoftware} code used to generate the simulation study results in Section~\ref{sec:simulations} is publicly available on~\href{https://github.com/FredericOuimetMcGill/MultivariateSzaszMirakyan}{GitHub}.

\section*{Funding}
\addcontentsline{toc}{section}{Funding}

Funding for this work was provided in part by the Natural Sciences and Engineering Research Council of Canada (NSERC) through Discovery Grant RGPIN-2026-04471 and Discovery Launch Supplement DGECR-2026-00449 awarded to Fr\'ed\'eric Ouimet. Additional support for this work was provided by Research Nova Scotia through a New Health Investigator Grant and by NSERC through a Discovery Grant (RGPIN-2019-07212), both awarded to Cindy Feng. Guanjie Lyu and Cindy Feng also acknowledge support from the Mitacs Accelerate program.

\addcontentsline{toc}{section}{References}

\bibliographystyle{plainnat}
\bibliography{bib_clean}

\end{document}

%% file: sim_model1.tex
\begin{table}[H]
\centering
\caption{Summary statistics of $\mathrm{ISE}_{S_\delta}$ for model (M1) over $N_{\mathrm{MC}} = 100$ replications.}
\label{tab:sim-model1}
{\setlength{\tabcolsep}{3pt}%
\begin{tabularx}{\linewidth}{@{}r*{9}{>{\centering\arraybackslash}X}@{}}
\toprule
& \multicolumn{2}{c}{Median ISE} & \multicolumn{2}{c}{IQR ISE} & \multicolumn{2}{c}{Mean ISE} & \multicolumn{2}{c}{Variance ISE} & \\
\cmidrule(lr){2-3}\cmidrule(lr){4-5}\cmidrule(lr){6-7}\cmidrule(lr){8-9}
$n$ & $F_n$ & $F_{\bm m^\star,n}$ & $F_n$ & $F_{\bm m^\star,n}$ & $F_n$ & $F_{\bm m^\star,n}$ & $F_n$ & $F_{\bm m^\star,n}$ & $\Delta_n$\\
\midrule
25  & 0.9064 & 0.6528 & 0.5587 & 0.4228 & 1.0558 & 0.8050 & 0.4331 & 0.4615 & 18.3356\\
50  & 0.4470 & 0.3515 & 0.2996 & 0.3086 & 0.5294 & 0.4188 & 0.0883 & 0.0847 & 20.3683\\
100  & 0.2411 & 0.2112 & 0.2287 & 0.2199 & 0.2930 & 0.2559 & 0.0276 & 0.0331 & 17.2015\\
200  & 0.1092 & 0.0884 & 0.0770 & 0.0817 & 0.1228 & 0.1065 & 0.0038 & 0.0049 & 19.0440\\
400  & 0.0573 & 0.0473 & 0.0436 & 0.0505 & 0.0673 & 0.0599 & 0.0013 & 0.0014 & 21.8015\\
\bottomrule
\end{tabularx}
}%
\end{table}

%% file: sim_model2.tex
\begin{table}[H]
\centering
\caption{Summary statistics of $\mathrm{ISE}_{S_\delta}$ for model (M2) over $N_{\mathrm{MC}} = 100$ replications.}
\label{tab:sim-model2}
{\setlength{\tabcolsep}{3pt}%
\begin{tabularx}{\linewidth}{@{}r*{9}{>{\centering\arraybackslash}X}@{}}
\toprule
& \multicolumn{2}{c}{Median ISE} & \multicolumn{2}{c}{IQR ISE} & \multicolumn{2}{c}{Mean ISE} & \multicolumn{2}{c}{Variance ISE} & \\
\cmidrule(lr){2-3}\cmidrule(lr){4-5}\cmidrule(lr){6-7}\cmidrule(lr){8-9}
$n$ & $F_n$ & $F_{\bm m^\star,n}$ & $F_n$ & $F_{\bm m^\star,n}$ & $F_n$ & $F_{\bm m^\star,n}$ & $F_n$ & $F_{\bm m^\star,n}$ & $\Delta_n$\\
\midrule
25  & 0.9139 & 0.7414 & 0.8360 & 0.7323 & 1.1851 & 0.9480 & 0.7979 & 0.8591 & 17.3323\\
50  & 0.4026 & 0.2971 & 0.3584 & 0.3328 & 0.5068 & 0.4287 & 0.1067 & 0.1515 & 14.3786\\
100  & 0.2058 & 0.1531 & 0.1587 & 0.1702 & 0.2481 & 0.2045 & 0.0226 & 0.0238 & 20.2655\\
200  & 0.1205 & 0.1031 & 0.0817 & 0.0795 & 0.1389 & 0.1225 & 0.0072 & 0.0081 & 19.2381\\
400  & 0.0557 & 0.0490 & 0.0477 & 0.0372 & 0.0714 & 0.0640 & 0.0020 & 0.0023 & 21.7787\\
\bottomrule
\end{tabularx}
}%
\end{table}

%% file: sim_mstar.tex
\begin{table}[H]
\centering
\caption{LSCV-selected smoothing levels. Reported entries summarize $\bm m^\star$ across Monte Carlo replications.}
\label{tab:sim-mstar}
\begin{tabularx}{\linewidth}{@{}r*{4}{>{\centering\arraybackslash}X}@{}}
\toprule
$n$ & Mean $m^\star_{\min}$ & Mean $m^\star_{\max}$ & Mean $(m^\star_{\min}/n^{2/3})$ & Mean $(m^\star_{\max}/n^{2/3})$\\
\midrule
25  & 5.6100 & 7.9150 & 0.6561 & 0.9257\\
50  & 7.1850 & 10.8800 & 0.5294 & 0.8016\\
100  & 11.1800 & 16.1400 & 0.5189 & 0.7492\\
200  & 17.4900 & 23.6750 & 0.5114 & 0.6923\\
400  & 28.2400 & 34.9450 & 0.5202 & 0.6437\\
\bottomrule
\end{tabularx}
\end{table}